\documentclass[11pt,a4paper]{article}
\usepackage{amsmath,amsfonts,amssymb,mathtools,amsthm,newlfont,graphicx,amscd,bbm,enumerate,galois,mathrsfs,xypic,geometry,bbm,xypic,titlesec}
\usepackage{tabularx,extarrows}
\usepackage{hyperref}
\usepackage{comment,tikz-cd,tikz,enumitem}
\usepackage{cleveref}

\usepackage{color}

\newcommand{\al}{\alpha}

\newcommand{\pprime}{{\prime\prime}}

\newcommand{\mov}{\mathrm{mov}}
\newcommand{\act}{\mathrm{act}}
\newcommand{\DR}{\mathrm{DR}}
\newcommand{\Coef}{\mathrm{Coef}}
\newcommand{\CR}{\mathrm{CR}}
\newcommand{\pt}{\mathrm{pt}}
\newcommand{\Spec}{\mathrm{Spec}\,}

\newcommand{\st}{\mathrm{st}}

\newcommand{\wk}{\mathrm{wk}}
\newcommand{\tp}{\mathrm{top}}
\newcommand{\str}{\mathrm{str}}
\newcommand{\DZ}{\mathrm{DZ}}

\newcommand{\ori}{\mathrm{ori}}
\newcommand{\tr}{\mathrm{Tr}}

\DeclareMathOperator{\res}{Res}

\newtheorem{thm}{Theorem}[section]

\newtheorem{rmk}[thm]{Remark}

\newtheorem{cor}[thm]{Corollary}
\newtheorem{lem}[thm]{Lemma}

\newtheorem{defn}[thm]{Definition}

\numberwithin{equation}{section}

\begin{document}

\title{Finite Group Reduction of the DR/DZ Hierarchies}
\author{Si-Qi Liu, Youjin Zhang, Tianwen Zhong}
\date{\today}
\maketitle

\begin{abstract}
We show that the finite group reduction of the Dubrovin-Zhang/Double Ramification hierarchy associated with a semisimple CohFT preserves its bihamiltonian structure and its tau structure. By applying this result to the orbifold Gromov-Witten theory with the target $\mathbb{P}^1_{2,2,2,2}$, we obtain two non-equivalent integrable hierarchies which can be associated with the Frobenius manifold structure on the Hurwitz space $M_{1;1}$, and we show that one of them is equivalent to the genus one topological recursion.

\end{abstract}


\tableofcontents

\section{Introduction}
Witten's conjecture \cite{Wit} led to a systematic study of the connection between integrable systems and enumerative geometry. The core idea is that enumerative invariants should be assembled into a generating series, and that the exponential of this generating series is expected to be a tau function of a certain integrable hierarchy. This conjecture was soon proved by Kontsevich in \cite{Kon}. Over the following decades, many individual cases have been studied. We list here some examples of enumerative theories whose associated integrable hierarchies can be explicitly described: the Gromov-Witten theory of $\mathbb{P}^1$ \cite{CDZ,DZ4,EK,EY,EYY,Get1} and its equivariant counterpart \cite{Get2,OP2}, the orbifold Gromov-Witten theory of $\mathbb{P}^1$-orbifolds \cite{CvdL,Joh,MST,MT}, and the FJRW theory of simple singularities \cite{FGM,FJR1,FJR2,FSZ,GM,HM,Wu,Wu2}. 

Witten's original conjecture admits the following generalizations. On the geometric side, Kontsevich and Manin introduced \textit{cohomological field theories} (CohFTs) to capture the axiomatic properties of enumerative geometry theories \cite{KM}. On the integrable system side, Dubrovin and the second-named author of the present paper constructed a bihamiltonian and tau-symmetric hierarchy (the DZ hierarchy) from a semisimple Frobenius manifold \cite{DZ3}. The tau structure of the DZ hierarchy gives rise to a tau function satisfying the Virasoro constraint. Suppose that this Frobenius manifold is given by the genus-zero topological part of a homogeneous CohFT with a flat identity, then this tau function coincides with the \textit{total descendent potential} of this CohFT due to the uniqueness of the solution to Dubrovin-Zhang's loop equation \cite{DZ3} and Givental-Teleman reconstruction theorem \cite{Giv,PPZ,Tel}.

Given a CohFT with a flat identity (not necessarily semisimple), one can construct a tau-symmetric Hamiltonian integrable hierarchy, called the double ramification hierarchy (the DR hierarchy), introduced in \cite{Bur} via the double ramification cycle on the moduli spaces of stable curves. If the CohFT is semisimple and homogeneous, the DR hierarchy admits a bihamiltonian structure and can be related to the DZ hierarchy as follows. The genus-zero topological part of the CohFT defines a semisimple Frobenius manifold, which in turn defines the DZ hierarchy. The DR hierarchy, after an explicitly constructed Miura-type transformation, coincides with the DZ hierarchy \cite{BLS}. Moreover, their bihamiltonian structures and tau structures agree \cite{BR2}.

In \cite{DZ3}, the authors summarized three axioms for integrable hierarchies: \textit{bihamiltonian structures}, \textit{tau structures}, and \textit{linearizable Virasoro symmetries}. The DZ hierarchy is the unique integrable hierarchy whose dispersionless limit coincides with the principal hierarchy and which satisfies all these axioms. Among these axioms, the bihamiltonian structure plays a distinguished role. One can associate with the bihamiltonian structure finitely many functions of one variable, called \textit{central invariants} \cite{LZ1}. A bihamiltonian integrable hierarchy possesses a tau structure if and only if all of these central invariants are constant functions \cite{DLZ3,FL}, and it satisfies the linearizable Virasoro symmetries axiom if and only if these constants are equal to $\frac{1}{24}$ \cite{LWZ2}.

If an integrable hierarchy satisfies all axioms except the linearizable Virasoro symmetries axiom, one can still define a tau function. Typical examples are the Drinfeld-Sokolov hierarchies of BCFG types. Their bihamiltonian structures have semisimple leading terms and their central invariants are constant functions but not all equal (see \cite{DLZ2}), which means that one cannot find CohFTs whose total descendent potentials give rise to tau functions of the BCFG Drinfeld-Sokolov hierarchies. However, Ruan, the first-named and the second-named author of the present paper are still able to prove a Witten conjecture type theorem in this case \cite{LRZ}. Their proof depends on the following two facts. The first one is that the total descendent potentials of the FJRW theories of ADE singularities are topological tau functions of the ADE Drinfeld-Sokolov hierarchies. Here the FJRW theories of ADE singularities define semisimple CohFTs, and the ADE Drinfeld--Sokolov hierarchies coincide with the corresponding DZ hierarchies. The second one is that both the enumerative geometry side and the integrable hierarchy side admit finite group symmetries. In analogy with the folding construction of the BCFG type Lie algebras from the ADE ones (see e.g. \cite{Kac}, \S 7.9), the BCFG hierarchies can be obtained from the ADE hierarchies by restricting to invariant flows. On the other hand, the FJRW theories admit enhanced symmetries by the same groups. The restrictions to fixed sectors yield new enumerative theories, which are \textit{partial} CohFTs rather than CohFTs. Finally, it can be verified that the restrictions on both sides match.

In this paper, we study the reduction properties of the DR/DZ hierarchies in a more general setting. Our starting point is a given semisimple CohFT with a finite group action, from which one can construct the corresponding DZ/DR hierarchies, which are referred to as the ambient DR/DZ hierarchies. The invariant flows of these hierarchies give rise to new integrable hierarchies, called the reduced hierarchies. Our main result is summarized in the following theorem, which combines Theorems~\ref{thm:main},~\ref{thm:reduceddrdz}, and~\ref{thm:taustructure}.
\begin{thm}
    The reduced DR/DZ hierarchies possess bihamiltonian structures and tau structures which are inherited from the ambient hierarchies. Moreover, the reduced DR hierarchy and the reduced DZ hierarchy are related by a Miura-type transformation preserving bihamiltonian structures and tau structures.
\end{thm}

Our strategy to arrive at the main result is to prove that the reduced DR hierarchy possesses a bihamiltonian structure and a tau structure first. Then we apply the string equation to lift the finite group action to the DR jet space and the DZ jet space respectively. We verify that the DR/DZ equivalence for the ambient hierarchies is equivariant with respect to the finite group action, so that it provides a DR/DZ equivalence for the reduced hierarchies. Finally, we use the reduced DR/DZ equivalence to translate everything back to the DZ side.

As an application, we construct two different finite group actions on the semisimple CohFT arising from the orbifold Gromov–Witten theory of $\mathbb{P}^1_{2,2,2,2}$. The first one is the $\mathfrak{S}_4$-action given by permuting the Poincar\'e dual of the twisted sectors, and the second one is a $\mathbb{Z}_2\times\mathbb{Z}_2$-action (see subsection \ref{sssec:equalcentralinvariants}). There are also two reduced integrable hierarchies, and the restrictions of the CohFTs give rise to two partial CohFTs as well. Both of them have the same genus-zero theory given by the Hurwitz space $M_{1;1}$ (see \cite{Dub}, Appendix~C and Lecture~5), and therefore the reduced hierarchies have the same dispersionless limit. However, the higher genus theories are not equivalent. One of them is equivalent to the topological recursion of the genus one spectral curve (see e.g. \cite{EO,GJZ}). We summarize the results in the following theorem.
\begin{thm}
    \begin{enumerate}[label=$\arabic*.$,ref=$\arabic*$]
        \item Under the $\mathfrak{S}_4$-action, the central invariants of the reduced integrable hierarchy are $\{\frac{1}{6}, \frac{1}{24}, \frac{1}{24}\}$. In particular, the corresponding partial CohFT satisfies a modified version of Virasoro constraints.
        \item Under the $\mathbb{Z}_2\times\mathbb{Z}_2$-action, the central invariants of the reduced integrable hierarchy are $\{\frac{1}{12}, \frac{1}{12}, \frac{1}{12}\}$. In particular, the corresponding partial CohFT is equivalent to the genus one topological recursion after a rescaling.
    \end{enumerate}
\end{thm}

It should be noted that although the latter case is equivalent to topological recursion, the corresponding partial CohFT satisfies modified Virasoro constraints rather than the usual ones (see Corollary~\ref{cor:nontrivial}). We plan to consider the Virasoro constraints for these two partial CohFTs elsewhere.

This paper is organized as follows. In Section~\ref{sec:diffpoly}, we summarize the definitions of differential polynomials and Hamiltonian structures, and we develop the corresponding theory in the presence of a finite group action. In Section~\ref{sec:CohFT}, we recall the definition of cohomological field theories and the construction of the DR hierarchy. In Section~\ref{sec:red}, we present a proof of the main result. In Section~\ref{sec:example}, we illustrate the usage of our main result, including the examples corresponding to $\mathbb{P}^1_{2,2,2,2}$.

\section{Preliminaries to differential polynomials with $\Gamma$-action}
\label{sec:diffpoly}
The following subsections are devoted to recalling the theory of differential polynomials and Hamiltonian structures in terms of super variables. For a detailed introduction, the readers may refer to \cite{LZ3}.

We fix an $N$-dimensional $\mathbb{C}$-vector space $H$ and a finite group $\Gamma$. Given a space or a ring on which $\Gamma$ acts, we denote by "$\gamma.$" the action of $\gamma\in\Gamma$. When we fix a (left) $\Gamma$-action on $H$ and a basis $\{\phi_\alpha\}$ of $H$, we use the notation $\bigl(\act(\gamma)^\alpha_\beta\bigr)$ to denote the matrix representation of $\gamma.\colon H\rightarrow H$, i.e. 
\begin{equation}
    \gamma.\phi_{\alpha}=\act(\gamma)_\alpha^\beta\phi_\beta,\quad\gamma\in\Gamma.
    \label{eq:actionmatrix}
\end{equation}
Here and henceforth, the Einstein summation convention for Greek indices is assumed unless otherwise stated.

\subsection{Differential polynomials and bihamiltonian structures}
\label{subsec:diffpoly}
Denote by $X=\Spec\mathrm{Sym}^\bullet \bigl(\mathrm{Hom}_\mathbb{C}(H,\mathbb{C})\bigr)$ the canonical scheme structure on $H$, and by $\mathfrak{X}$ the formal completion of $X$ along the origin. Let $\hat{\mathfrak{X}}=\Pi (T^* \mathfrak{X})$ be the cotangent bundle of $\mathfrak{X}$ with fiber's parity reversed, and $J^{\infty}(\hat{\mathfrak{X}})$ be the infinite jet space of $\hat{\mathfrak{X}}$. The formal (super-)schemes $\mathfrak{X}$, $\hat{\mathfrak{X}}$, and $J^\infty(\hat{\mathfrak{X}})$ are supported at a single point. So all their information is encoded in their coordinate rings. We will use the notation $\hat{\mathscr{A}}$ for the coordinate ring of $J^{\infty}(\hat{\mathfrak{X}})$ instead of $\mathcal{O}_{J^{\infty}(\hat{\mathfrak{X}})}$. Elements in $\hat{\mathscr{A}}$ are called differential polynomials.

The constructions mentioned above can be described locally. In what follows, a (formal) coordinate system of $\mathfrak{X}$ is a local isomorphism of $\mathbb{C}$-algebras:
\begin{equation*}
    \varphi\colon\mathbb{C}[[v^1,\ldots,v^N]]\xrightarrow{\cong} \mathcal{O}_\mathfrak{X}
\end{equation*}
We will omit $\varphi$ whenever no confusion can arise. Similarly, we can consider coordinates of $\hat{\mathfrak{X}}$ and $J^\infty(\hat{\mathfrak{X}})$. A coordinate system $(v^\alpha)$ of $\mathfrak{X}$ induces a coordinate system $(v^\alpha;\rho_\alpha)$ of $\hat{\mathfrak{X}}$ and a coordinate system $(v^{\alpha,s};\rho_\alpha^s)$ of $J^\infty(\hat{\mathfrak{X}})$. Here $\rho_\alpha$ are coordinate functions on the fiber of $T^*\mathfrak{X}$, and the variables $\rho_\alpha^s$ are odd, i.e., 
\begin{equation*}
    \rho^s_{\alpha}\rho^t_{\beta}+\rho^t_{\beta}\rho^s_{\alpha}=0.
\end{equation*}
We also use the convention
\begin{equation*}
    v^{\alpha,0}=v^\alpha,\quad\rho_\alpha^0=\rho_\alpha.
\end{equation*}
Later, when we deal with several coordinate systems, we will use the notation $\hat{\mathscr{A}}_v$ to indicate that we have identified $\hat{\mathscr{A}}$ with $\mathbb{C}[[v^{\alpha,s},\rho_\alpha^s]]$.

A coordinate change from $(\widetilde{v}^\alpha;\widetilde{\varphi})$ to $(v^\alpha;\varphi)$ on $\mathfrak{X}$ is a local isomorphism
\begin{equation*}
    f\colon\mathbb{C}[[\widetilde{v}^\alpha]]\rightarrow\mathbb{C}[[v^\alpha]]    
\end{equation*}
such that
\begin{equation*}
    \varphi\circ f=\widetilde{\varphi}.
\end{equation*}
In what follows, we will simply write $\widetilde{v}^\alpha(v)$ for $f(\widetilde{v}^\alpha)$ and $v^\alpha(\widetilde{v})$ for $f^{-1}(v^\alpha)$. Under this coordinate change, the induced coordinate change from $(v^{\alpha,s};\rho_\alpha^s)$ to $(\widetilde{v}^{\alpha,s};\widetilde{\rho}_\alpha^s)$ on $J^\infty(\hat{\mathfrak{X}})$ can be recursively determined as follows
\begin{equation*}
    \widetilde{v}^{\alpha,s+1}=\sum_{i=0}^s\frac{\partial \widetilde{v}^{\alpha,s}}{\partial v^{\beta,i}}v^{\beta,i+1},\quad \widetilde{\rho}_\alpha^{s}=\sum_{i=0}^s\binom{s}{i}\frac{\partial v^{\beta,i}}{\partial \widetilde{v}^\alpha}\rho_\beta^{s-i},\quad s\geq 0.
\end{equation*}
There is a global vector field on $J^\infty(\hat{\mathfrak{X}})$ defined as follows
\begin{equation}
    \partial=\sum_{s\geq 0}\left(v^{\alpha,s+1}\frac{\partial}{\partial v^{\alpha,s}}+\rho^{s+1}_\alpha\frac{\partial}{\partial \rho_\alpha^s}\right).
    \label{eq:partial}
\end{equation}

Fix a coordinate system $(v^{\alpha,s};\rho_\alpha^s)$, we can define two gradations on $\hat{\mathscr{A}}$. The standard gradation is defined by assigning degree $s$ to $v^{\alpha,s}$ and $\rho_\alpha^s$. The corresponding graded piece of degree $d$ will be denoted by $\hat{\mathscr{A}}_d$. Sometimes we use a formal variable $\varepsilon$ to indicate the standard gradation, i.e. we assign $\varepsilon$ a degree $-1$ and identify $\hat{\mathscr{A}}_v$ with the subring of degree zero elements of
\begin{equation*}
    \mathbb{C}[[v^\alpha,\rho_\alpha]][v^{\alpha,s+1},\rho_\alpha^{s+1}\mid s\geq 0][[\varepsilon]].
\end{equation*}
The second gradation is called the super gradation. It is defined by assigning degree $1$ to $\rho_\alpha^s$ and $0$ to $v^{\alpha,s}$. The associated degree $p$ part is denoted by $\hat{\mathscr{A}}^p$. We also use the notation $\hat{\mathscr{A}}_d^p=\hat{\mathscr{A}}_d\cap\hat{\mathscr{A}}^p$. Note that these two gradations are independent of the choice of coordinates on $\mathfrak{X}$.

By local functionals, we mean elements of the quotient space $\hat{\mathscr{F}}=\hat{\mathscr{A}}/\partial\hat{\mathscr{A}}$. The class with representative $f\in\hat{\mathscr{A}}$ is denoted by $\int f$. Since $\partial$ is homogeneous in the sense that $\partial(\hat{\mathscr{A}}^p_d)\subset\hat{\mathscr{A}}^p_{d+1}$, there are also two induced gradations on $\hat{\mathscr{F}}$. The graded pieces will be denoted by $\hat{\mathscr{F}}_d$, $\hat{\mathscr{F}}^p$, and $\hat{\mathscr{F}}^p_d$.

For an element $f\in\hat{\mathscr{A}}$, we define its variational derivatives by
\begin{equation*}
    \frac{\delta f}{\delta v^\alpha}=\sum_{s\geq 0}(-\partial)^s\frac{\partial f}{\partial v^{\alpha,s}},\quad\frac{\delta f}{\delta\rho_\alpha}=\sum_{s\geq 0}(-\partial)^s\frac{\partial f}{\partial\rho_\alpha^s}.
\end{equation*}
They are homogeneous in the sense that 
\begin{equation*}
    \frac{\delta}{\delta v^\alpha}\left(\hat{\mathscr{A}}^p_d\right)\subset\hat{\mathscr{A}}^p_d,\quad\frac{\delta}{\delta\rho_\alpha}\left(\hat{\mathscr{A}}^p_d\right)\subset\hat{\mathscr{A}}^{p-1}_d.
\end{equation*}
One can check that variational derivatives annihilate $\partial\hat{\mathscr{A}}$, and therefore descend to $\hat{\mathscr{F}}$. We denote the induced variational derivatives on $\hat{\mathscr{F}}$ by the same notation:
\begin{equation*}
    \frac{\delta}{\delta v^\alpha},\,\frac{\delta}{\delta\rho_\alpha}\colon\hat{\mathscr{F}}\rightarrow\hat{\mathscr{A}}.
\end{equation*}

We define the Schouten-Nijenhuis bracket on $\hat{\mathscr{F}}$ by the following expression on homogeneous components:
\begin{equation*}
    [-,-]\colon\hat{\mathscr{F}}^p\times\hat{\mathscr{F}}^q\rightarrow\hat{\mathscr{F}}^{p+q-1},\quad (P,Q)\mapsto \int\frac{\delta P}{\delta \rho_\alpha}\frac{\delta Q}{\delta v^\alpha}+(-1)^p\frac{\delta P}{\delta v^\alpha}\frac{\delta Q}{\delta \rho_\alpha}.
\end{equation*}

\begin{defn}
    A Hamiltonian structure is a local functional $P\in\hat{\mathscr{F}}^2_{\geq 1}$ satisfying $[P,P]=0$. A bihamiltonian structure is a pair $(P,Q)$ of Hamiltonian structures satisfying $[P,Q]=0$.
\end{defn}

For any $P\in\hat{\mathscr{F}}^2_{\geq 1}$, we can define a matrix differential operator $\mathscr{P}=\left(\sum_{s\geq 0}\mathscr{P}^{\alpha\beta}_s\partial^s\right)$ by
\begin{equation*}
    \sum_{s\geq 0}\mathscr{P}_s^{\alpha\beta}\rho_\beta^s=\frac{\delta P}{\delta\rho_\alpha}.
\end{equation*}
Note that $P=\int\frac{1}{2}\rho_\alpha\frac{\delta P}{\delta\rho_\alpha}$. The operator $\mathscr{P}$ satisfies the following antisymmetry property
\begin{equation}
    \sum_{s\geq 0}\mathscr{P}^{\alpha\beta}_s\partial^s=\sum_{s\geq 0}(-1)^{s+1}\partial^s\circ\mathscr{P}^{\beta\alpha}_s,
    \label{eq:antisym}
\end{equation}
and the leading term property
\begin{equation}
    \mathscr{P}^{\alpha\beta}_0\in\hat{\mathscr{A}}^0_{\geq 1}.
    \label{eq:leadingterm}
\end{equation}

We can define an antisymmetric bilinear bracket as follows
\begin{equation}
    \{-,-\}_P\colon\hat{\mathscr{F}}^0\times\hat{\mathscr{F}}^0\rightarrow\hat{\mathscr{F}}^0,\quad (F,G)\mapsto\{F,G\}_P=\int\left(\frac{\delta F}{\delta v^\alpha}\right)\sum_{s\geq 0}\mathscr{P}_s^{\alpha\beta}\partial^s\left(\frac{\delta G}{\delta v^\beta}\right).
    \label{eq:poisson}
\end{equation}
The property \eqref{eq:antisym} implies that $\{F,G\}_P=-[F,[P,G]]$. The following theorem shows that the condition $[P,P]=0$ is equivalent to the Jacobi identity of $\{-,-\}_P$. 
\begin{thm}[\cite{LZ2}, Corollary 2.4.5]
    The Schouten-Nijenhuis bracket makes $\hat{\mathscr{F}}$ a Lie superalgebra, i.e. for any $P\in\hat{\mathscr{F}}^p,Q\in\hat{\mathscr{F}}^q,R\in\hat{\mathscr{F}}^r$, we have
    \begin{enumerate}[label=$\arabic*.$,ref=$\arabic*$]
        \item $[P,Q]=(-1)^{pq}[Q,P]$,
        \item $(-1)^{pr}[[P,Q],R]+(-1)^{qp}[[Q,R],P]+(-1)^{rq}[[R,P],Q]=0$.
    \end{enumerate}
\end{thm}
So once a coordinate system is fixed, a Hamiltonian structure can be identified with a matrix differential operator so that it satisfies \eqref{eq:antisym}, \eqref{eq:leadingterm} and the bracket \eqref{eq:poisson} satisfies the Jacobi identity.

Define Miura-type transformations to be coordinate changes of $J^\infty(\hat{\mathfrak{X}})$ that commute with $\partial$. They form a group. Given two coordinate systems $(v^{\alpha,s};\rho_\alpha^s)$ and $(\widetilde{v}^{\alpha,s};\widetilde{\rho}_\alpha^s)$, a Miura-type transformation relating them is of the following form:
\begin{equation}
    v^{\alpha,s}\mapsto \widetilde{v}^{\alpha,s}=\sum_{d\geq 0}\partial^s\bigl(f^{\alpha}_d(v)\bigr),\quad \rho^s_\alpha\mapsto \widetilde{\rho}_\alpha^s=\sum_{i\geq 0}(-1)^i\partial^{i+s}\left(\frac{\partial v^\beta}{\partial \widetilde{v}^{\alpha,i}}\rho_\beta\right),\quad f^\alpha_d(v)\in\hat{\mathscr{A}}^0_{v,d}.
    \label{eq:miura}
\end{equation}
In particular, we have
\begin{equation}
    f^\alpha_0\bigg\vert_{v^\bullet\rightarrow 0}=0,\quad\det \left(\frac{\partial f^\alpha_0}{\partial v^\beta}\right)\bigg\vert_{v^\bullet\rightarrow 0}\neq 0.
    \label{eq:localiso}
\end{equation}
Note that a Miura-type transformation is determined by a collection of differential polynomials $\{f^\alpha_d\}$ satisfying \eqref{eq:localiso}.
It is said to be of the 2nd kind if $f^\alpha_0(v)=v^\alpha$. Miura-type transformations are shown in \cite{LZ2} to preserve the Schouten-Nijenhuis bracket.

\subsection{The fixed locus of differential polynomials under $\Gamma$-action}
Suppose that there is a $\Gamma$-action on $H$. Choose a basis $\{\phi_\alpha\}$ for $H$. Denote by $\{\phi^\alpha\}$ the corresponding dual basis of $\mathrm{Hom}_\mathbb{C}(H,\mathbb{C})$. Then we have the following isomorphism
\begin{equation*}
    \mathbb{C}[u^1,\ldots\,u^N]\rightarrow\mathcal{O}_X(X),\quad u^\alpha\mapsto\phi^\alpha.
\end{equation*}
After taking formal completion, it gives rise to a coordinate system $(u^\alpha)$ on $\mathfrak{X}$, and therefore on $\hat{\mathfrak{X}}$ and $J^\infty(\hat{\mathfrak{X}})$.
\begin{defn}
    A coordinate system obtained in this way is called a dual coordinate system with respect to $\{\phi_\alpha\}$. We will use the notation $(u^{\alpha,s};\theta_\alpha^s)$ for the dual coordinate system on $J^\infty(\hat{\mathfrak{X}})$.
    \label{defn:dual}
\end{defn}

The $\Gamma$-action on $H$ determines an algebraic action on $X$, which lifts naturally to $\mathfrak{X}$ and $\hat{\mathfrak{X}}$. By requiring that the $\Gamma$-action commutes with the $\partial$-action, we obtain an induced $\Gamma$-action on $\hat{\mathscr{A}}$. It can be described in the dual coordinate system $(u^{\alpha,s};\theta_\alpha^s)$ by
\begin{equation}
    \gamma.u^{\alpha,s}=\act(\gamma^{-1})^\alpha_\beta u^{\beta,s},\quad\gamma.\theta_\alpha^s=\act(\gamma)_\alpha^\beta\theta_\beta^s,\quad\gamma\in\Gamma.
    \label{eq:action}
\end{equation}

\begin{defn}
    A coordinate system $(v^{\alpha,s};\rho_\alpha^s)$ is called a $(\Gamma$-$)$linearized coordinate system with respect to the basis $\{\phi_\alpha\}$ if the $\Gamma$-action on $\hat{\mathscr{A}}$ satisfies \eqref{eq:action}.
\end{defn}

Denote by $X^\Gamma$ the fixed locus of $\Gamma$-action on $X$. The constructions in Section~\ref{subsec:diffpoly} can be carried over to $X^\Gamma$. When referring to a construction that applies to both $X$ and $X^\Gamma$, we will attach a superscript $(-)^\Gamma$ to indicate those associated with $X^\Gamma$. For instance, we will use the following notations $\mathcal{O}_{\mathfrak{X}^\Gamma},J^\infty(\widehat{\mathfrak{X}^\Gamma}),\partial^\Gamma,\hat{\mathscr{A}}^\Gamma,\hat{\mathscr{F}}^\Gamma$.
Let $\iota:X^\Gamma\hookrightarrow X$ be the inclusion. It induces a homomorphism of $\mathbb{C}$-superalgebras
\begin{equation*}
    \iota^*\colon\hat{\mathscr{A}}\rightarrow\hat{\mathscr{A}}^\Gamma.
\end{equation*}
Consider the two-sided ideal $\hat{\mathscr{I}}$ of $\hat{\mathscr{A}}$ generated by
\begin{equation*}
    f-\gamma.f,\quad\gamma\in\Gamma,f\in\hat{\mathscr{A}}.
\end{equation*}
Then $\iota^*$ is the quotient map
\begin{equation*}
    \hat{\mathscr{A}}\rightarrow\hat{\mathscr{A}}^\Gamma\cong\hat{\mathscr{A}}/\hat{\mathscr{I}}.
\end{equation*}

A finite dimensional $\mathbb{C}$-representation of finite groups is always semisimple (see e.g. \cite{Ser}). So we have a direct sum decomposition of $\Gamma$-representations
\begin{equation}
    H=H^\Gamma\oplus H^\mov.
    \label{eq:decom}
\end{equation}
where 
\begin{equation*}
    H^\Gamma=\{e\in H\mid {\gamma}.e=e,\,\forall \gamma\in \Gamma\}
\end{equation*}
is the fixed part and $H^\mov$ is the direct sum of nontrivial representations under the irreducible decomposition. We can choose basis $\{\phi_{\alpha^\prime}\}_{\alpha^\prime\in I}$ for $H^\Gamma$ and $\{\phi_{\alpha^\pprime}\}_{\alpha^\pprime\in J}$ for $H^\mov$ respectively, where
\begin{equation*}
    I=\{\alpha^\prime\mid 1\leq\alpha^\prime\leq \dim_\mathbb{C} H^\Gamma\},\quad J=\{\alpha^\pprime\mid 1+\dim_\mathbb{C} H^\Gamma\leq\alpha^\pprime\leq N\}.
\end{equation*}
They give a basis $\{\phi_\alpha\}$ of $H$. With respect to this basis, the matrices $\act(\gamma)$ satisfy
\begin{equation}
    \act(\gamma)^{\beta}_{\alpha^\prime}=\delta_{\alpha^\prime}^\beta,\quad \act(\gamma)^{\beta^\prime}_{\alpha^\pprime}=0,\quad \forall\alpha^\prime,\beta^\prime\in I,\,\alpha^\pprime\in J,\,1\leq\beta\leq N.
    \label{eq:simplifyactionmatrix}
\end{equation}
In a $\Gamma$-linearized coordinate system $(v^{\alpha,s};\rho_{\alpha}^s)$ with respect to this basis, the ideal $\hat{\mathscr{I}}$ is generated by
\begin{equation*}
    v^{\alpha^\pprime,s},\rho_{\alpha^\pprime}^s,\quad\alpha^\pprime\in J,\,s\geq 0.
\end{equation*}
In what follows, we always work with a basis of $H$ chosen with respect to the decomposition~\eqref{eq:decom}.

Note that $\iota^*\circ\partial=\partial^\Gamma\circ\iota^*$. The homomorphism $\iota^*$ also descends to a $\mathbb{C}$-linear map between spaces of local functionals
\begin{equation*}
    \bar{\iota}^*\colon\hat{\mathscr{F}}\rightarrow\hat{\mathscr{F}}^\Gamma.
\end{equation*}
Denote by $[-,-]^\Gamma$ the Schouten-Nijenhuis bracket on $\hat{\mathscr{F}}^\Gamma$. In general, the map $\bar{\iota}^*$ does not preserve the Schouten-Nijenhuis bracket, let alone Hamiltonian structures. However, the following result holds.
\begin{lem}
    Let $(v^{\alpha,s};\rho_{\alpha}^s)$ be a $\Gamma$-linearized coordinate system for $J^\infty(\hat{\mathfrak{X}})$, and $\widetilde{Q}\in\hat{\mathscr{A}}^q$ be a differential polynomial  satisfying
    \begin{equation}
        \iota^*\left(\frac{\partial\widetilde{Q}}{\partial v^{\alpha^\pprime,s}}\right)=\iota^*\left(\frac{\partial\widetilde{Q}}{\partial\rho_{\alpha^\pprime}^s}\right)=0,\quad\forall\alpha^\pprime\in J,\, s\geq 0.
        \label{eq:main}
    \end{equation}
    Then $[\bar{\iota}^*(-),\bar{\iota}^*(Q)]^\Gamma=\bar{\iota}^*([-,Q])$ in $\hat{\mathscr{F}}^\Gamma$, where $Q=\int\widetilde{Q}$.
    \label{lem:main}
\end{lem}

\begin{proof}
The second assumption implies
\begin{equation*}
    \iota^*\left(\frac{\delta Q}{\delta v^{\alpha^\pprime}}\right)=\iota^*\left(\frac{\delta Q}{\delta\rho_{\alpha^\pprime}}\right)=0,\quad\forall\alpha^\pprime\in J.
\end{equation*}
Note that we have the following compatible results of $\iota^*,\bar{\iota}^*$ with variational derivatives:
\begin{equation*}
    \iota^*\circ\frac{\delta}{\delta v^{\alpha^\prime}}=\frac{\delta}{\delta v^{\alpha^\prime}}\circ\bar{\iota}^*,\quad\iota^*\circ\frac{\delta}{\delta\rho_{\alpha^\prime}}=\frac{\delta}{\delta\rho_{\alpha^\prime}}\circ\bar{\iota}^*,\quad\forall\alpha^\prime\in I.
\end{equation*}
For any $P\in\hat{\mathscr{F}}^p$, we have
\begin{align*}
    &[\bar{\iota}^*(P),\bar{\iota}^*(Q)]^\Gamma\\
    =&\int\frac{\delta\bar{\iota}^*(P)}{\delta\rho_{\alpha^\prime}}\frac{\delta\bar{\iota}^*(Q)}{\delta v^{\alpha^\prime}}+(-1)^{p}\frac{\delta\bar{\iota}^*(P)}{\delta v^{\alpha^\prime}}\frac{\delta\bar{\iota}^*(Q)}{\delta\rho_{\alpha^\prime}}\\
    =&\int\iota^*\left(\frac{\delta P}{\delta\rho_{\alpha^\prime}}\right)\iota^*\left(\frac{\delta Q}{\delta v^{\alpha^\prime}}\right)+(-1)^{p}\iota^*\left(\frac{\delta P}{\delta v^{\alpha^\prime}}\right)\iota^*\left(\frac{\delta Q}{\delta\rho_{\alpha^\prime}}\right)\\
    =&\int\iota^*\left(\frac{\delta P}{\delta\rho_{\alpha}}\right)\iota^*\left(\frac{\delta Q}{\delta v^{\alpha}}\right)+(-1)^{p}\iota^*\left(\frac{\delta P}{\delta v^{\alpha}}\right)\iota^*\left(\frac{\delta Q}{\delta\rho_{\alpha}}\right)\\
    =&\,\overline{\iota}^*\left([P,Q]\right).
\end{align*}
Here and henceforth, whenever two repeated indices with superscript $(-)^\prime$ (resp. $(-)^\pprime$) appear, the summation is taken over the index set $I$ (resp. $J$) instead of $\{1,2,\ldots,N\}$.
\end{proof}

\begin{defn}
    Let $(v^{\alpha,s};\rho_{\alpha}^s)$ and $(\widetilde{v}^{\alpha,s};\widetilde{\rho}_{\alpha}^s)$ be two coordinate systems. A Miura-type transformation of the 2nd kind
    \begin{equation*}
        v^{\alpha}\mapsto\widetilde{v}^\alpha=v^\alpha+\sum_{d\geq 1}f^\alpha_d(v),\quad f^\alpha_d\in\hat{\mathscr{A}}^0_{v,d}
    \end{equation*}
    is called $(\Gamma$-$)$linearized if
    \begin{equation}
        \gamma.\Bigl(\sum_{d\geq 1}f^\beta_d\Bigr)=\act(\gamma^{-1})_\alpha^\beta\Bigl(\sum_{d\geq 1}f^\alpha_d\Bigr),\quad\forall\gamma\in\Gamma,\,\beta\in\{1,2,\dots,N\}.
        \label{eq:linmiura}
    \end{equation}
\end{defn}

\begin{lem}
    $\Gamma$-linearized Miura-type transformations form a group.
\end{lem}

\begin{proof}
    We prove here that the inverse of a $\Gamma$-linearized Miura-type transformation is $\Gamma$-linearized. The fact that a composition of two $\Gamma$-linearized Miura-type transformations is $\Gamma$-linearized can be proved similarly. Let 
    \begin{equation*}
        v^{\alpha}\mapsto\widetilde{v}^\alpha=v^\alpha+\sum_{d\geq 1}f^\alpha_d(v),\quad f^\alpha_d\in\hat{\mathscr{A}}^0_{v,d}
    \end{equation*}
    be a $\Gamma$-linearized Miura-type transformation with inverse
    \begin{equation*}
        \widetilde{v}^\alpha\mapsto v^{\alpha}= \widetilde{v}^\alpha+\sum_{d\geq 1}\widetilde{f}^\alpha_d(\widetilde{v}),\quad\widetilde{f}^\alpha_d\in\hat{\mathscr{A}}^0_{\widetilde{v},d}.
    \end{equation*}
    Then we have
    \begin{equation*}
        0=v^\alpha(\widetilde{v})+\sum_{d\geq 1}f^\alpha_d\bigl(v(\widetilde{v})\bigr)-\widetilde{v}^\alpha=\sum_{d\geq 1}\widetilde{f}^\alpha_d(\widetilde{v})+\sum_{d\geq 1}f^\alpha_d\bigl(v(\widetilde{v})\bigr),
    \end{equation*}
    and therefore,
    \begin{align*}
        &\gamma.\left(\sum_{d\geq 1}\widetilde{f}^\beta_d(\widetilde{v})\right)=-\gamma.\left(\sum_{d\geq 1}f^\beta_d\bigl(v(\widetilde{v})\bigr)\right)\\
        =&-\act(\gamma^{-1})_\alpha^\beta\left(\sum_{d\geq 1}f^\alpha_d\bigl(v(\widetilde{v})\bigr)\right)
        =\act(\gamma^{-1})_\alpha^\beta\left(\sum_{d\geq 1}\widetilde{f}^\alpha_d(\widetilde{v})\right).
    \end{align*}
    The lemma is proved.
\end{proof}

\begin{lem}
    Let $(v^{\alpha,s};\rho_{\alpha}^s)$, $(\widetilde{v}^{\alpha,s};\widetilde{\rho}_{\alpha}^s)$ be two coordinate systems, and
    \begin{equation*}
        v^{\alpha}\mapsto\widetilde{v}^\alpha=v^\alpha+\sum_{d\geq 1}\widetilde{f}^\alpha_d(v),\quad \widetilde{f}^\alpha_d\in\hat{\mathscr{A}}^0_{v,d}
    \end{equation*}
    be a $\Gamma$-linearized Miura-type transformation. If one of them is a $\Gamma$-linearized coordinate system, then so is the other.
    \label{lem:linearizedmiura}
\end{lem}

\begin{proof}
    Since $\Gamma$-linearized Miura-type transformations form a group, we may assume that $(v^{\alpha,s};\rho_{\alpha}^s)$ is a $\Gamma$-linearized coordinate system. The equation \eqref{eq:linmiura} shows that 
    \begin{align*}
        &\gamma.\widetilde{v}^\alpha=\gamma.\bigl(v^\alpha(\widetilde{v})\bigr)+\gamma.\left(\sum_{d\geq 1}\widetilde{f}^\alpha_d\bigl(v(\widetilde{v})\bigr)\right)\\
        =&\,\act(\gamma^{-1})^\alpha_\beta v^\beta(\widetilde{v})+\act(\gamma^{-1})_\beta^\alpha\left(\sum_{d\geq 1}\widetilde{f}^\beta_d\bigl(v(\widetilde{v})\bigr)\right)=\act(\gamma^{-1})_\beta^\alpha\widetilde{v}^\beta.
    \end{align*}
    In particular, we have
    \begin{equation*}
        \gamma.\left(\frac{\partial v^\beta(\widetilde{v})}{\partial \widetilde{v}^{\alpha,s}}\right)=\act(\gamma^{-1})^\beta_\mu\act(\gamma)^\nu_\alpha\frac{\partial v^\mu(\widetilde{v})}{\partial\widetilde{v}^{\nu,s}}
    \end{equation*}
    and therefore,
    \begin{align*}
        &\gamma.\widetilde{\rho}_{\alpha}=\sum_{s\geq 0}(-\partial)^s\left(\gamma.\left(\frac{\partial v^\beta}{\partial \widetilde{v}^{\alpha,s}}\right)\gamma.\rho_\beta\right)\\
        =&\,\sum_{s\geq 0}(-\partial)^s\left(\act(\gamma^{-1})^\beta_\mu\act(\gamma)_\alpha^\nu\frac{\partial v^\mu}{\partial\widetilde{v}^{\nu,s}}\act(\gamma)_\beta^\zeta\rho_\zeta\right)\\
        =&\,\act(\gamma)_\alpha^\nu\left(\sum_{s\geq 0}(-\partial)^s\left(\frac{\partial v^\mu}{\partial\widetilde{v}^{\nu,s}}\rho_\mu\right)\right)=\act(\gamma)_\alpha^\nu\widetilde{\rho}_\nu.
    \end{align*}
    This shows that $(\widetilde{v}^{\alpha,s};\widetilde{\rho}_{\alpha}^s)$ is also a $\Gamma$-linearized coordinate system.
\end{proof}
\begin{rmk}
    For a $\Gamma$-linearized Miura-type transformation
    \begin{equation*}
        v^\alpha\mapsto\widetilde{v}^\alpha=v^\alpha+\sum_{d\geq 1}f^\alpha_d
    \end{equation*}
    between $\Gamma$-linearized coordinate systems, the differential polynomials $\iota^*(f^{\alpha^\prime}_d),\alpha^\prime\in I$ define a Miura-type transformation of the 2nd kind on $\hat{\mathscr{A}}^\Gamma$ and fit into the following commutative diagram
    \begin{equation*}
        \begin{tikzcd}
            \hat{\mathscr{A}}_{\widetilde{v}} \arrow[r, "\cong"] \arrow[d, "\iota^*"] 
            & \hat{\mathscr{A}}_{v} \arrow[d, "\iota^*"] \\
            \hat{\mathscr{A}}^\Gamma_{\widetilde{v}} \arrow[r, "\cong"] 
            & \hat{\mathscr{A}}^\Gamma_{v}
        \end{tikzcd}
    \end{equation*}
    \label{rmk:inducedmiura}
\end{rmk}

The following trick, originating from Proposition 2.6 of \cite{LRZ}, will be used repeatedly throughout this paper.
\begin{lem}
    \label{lem:trick}
    Fix a $\Gamma$-action on $H$ and a positive integer $n$. We choose a basis for $H$ so that the equalities in \eqref{eq:simplifyactionmatrix} hold. Suppose that there is a collection of complex numbers
    \begin{equation*}
        c_{\alpha_1,s_1;\cdots;\alpha_n,s_n}\in\mathbb{C},\quad 1\leq\alpha_i\leq N,\,s_i\geq 0,
    \end{equation*}
    satisfying
    \begin{equation*}
        \act(\gamma)_{\alpha_1}^{\beta_1}\cdots\act(\gamma)_{\alpha_n}^{\beta_n}c_{\beta_1,s_1;\cdots;\beta_n,s_n}=c_{\alpha_1,s_1;\cdots;\alpha_n,s_n},\quad\forall\gamma\in\Gamma.
    \end{equation*}
    Then we have
    \begin{equation*}
        c_{\alpha^\pprime,s;\alpha_1^\prime,s_1;\cdots;\alpha_{n-1}^\prime,s_{n-1}}=0,\quad\forall\alpha^\pprime\in J,\,\alpha_i^\prime\in I,\,s,s_i\geq 0.
    \end{equation*}
\end{lem}

\begin{proof}
    Note that for all $\alpha^\pprime\in J$, we have
    \begin{equation*}
        \frac{1}{|\Gamma|}\sum_{\gamma\in\Gamma}\gamma.\phi_{\alpha^\pprime}=\frac{1}{|\Gamma|}\sum_{\gamma\in\Gamma}\act(\gamma)_{\alpha^\pprime}^\beta\phi_\beta\in H^\Gamma\cap H^\mov=\{0\}.
    \end{equation*}
    This implies that for all $1\leq\beta\leq N$, we have
    \begin{equation*}
        \frac{1}{|\Gamma|}\sum_{\gamma\in\Gamma}\act(\gamma)_{\alpha^\pprime}^\beta=0.
    \end{equation*}
    Hence, 
    \begin{align*}
        &c_{\alpha^\pprime,s;\alpha_1^\prime,s_1;\cdots;\alpha_{n-1}^\prime,s_{n-1}}=\frac{1}{|\Gamma|}\sum_{\gamma\in\Gamma}\act(\gamma)_{\alpha^\pprime}^\beta\act(\gamma)_{\alpha^\prime_1}^{\beta_1}\cdots\act(\gamma)_{\alpha^\prime_{n-1}}^{\beta_{n-1}}c_{\beta,s;\beta_1,s_1;\cdots;\beta_{n-1},s_{n-1}}\\
        =&\Bigl(\frac{1}{|\Gamma|}\sum_{\gamma\in\Gamma}\act(\gamma)_{\alpha^\pprime}^\beta \Bigr)c_{\beta,s;\alpha_1^\prime,s_1;\cdots;\alpha_{n-1}^\prime,s_{n-1}}=0.
    \end{align*}
    The lemma is proved.
\end{proof}

As an application, we can show that the property \eqref{eq:main} is invariant under a Miura-type transformation of the 2nd kind between linearized coordinate systems.
\begin{lem}
    Let 
    \begin{equation*}
        v^\alpha\mapsto \widetilde{v}^\alpha=v^\alpha+\sum_{d\geq 1}\widetilde{f}_d^\alpha(v),\quad\widetilde{f}_d^\alpha\in\hat{\mathscr{A}}^0_{v,d}
    \end{equation*}
    be a Miura-type transformation of the 2nd kind relating two $\Gamma$-linearized coordinate systems $(v^{\alpha,s};\rho^s_\alpha)$ and $(\widetilde{v}^{\alpha,s};\widetilde{\rho}^s_\alpha)$, and let $Q\in\hat{\mathscr{A}}$ be a differential polynomial satisfying \eqref{eq:main} in the coordinate system $(v^{\alpha,s};\rho^s_\alpha)$. Then $Q$ also satisfies \eqref{eq:main} in the coordinate system $(\widetilde{v}^{\alpha,s};\widetilde{\rho}^s_\alpha)$. 
    \label{lem:trickapply}
\end{lem}

\begin{proof}
    For a fixed pair $(\beta^\prime,t)\in I\times\mathbb{Z}_{\geq 0}$, we can write $v^{\beta^\prime,t}$ as a formal power series in $\widetilde{v}^{\alpha,s}$ as follows
    \begin{equation*}
        v^{\beta^\prime,t}=\sum_{n\geq 0}\frac{1}{n!}\sum_{s_1,\ldots,s_n\geq 0} c^{\beta^\prime,t}_{\alpha_1,s_1;\cdots;\alpha_n,s_n}\widetilde{v}^{\alpha_1,s_1}\cdots \widetilde{v}^{\alpha_n,s_n}.
    \end{equation*}
    These coefficients are uniquely determined if we require that each $c^{\beta^\prime,t}_{\alpha_1,s_1;\cdots;\alpha_n,s_n}$ is invariant under permutations of the pairs $(\alpha_i,s_i)$. Since $(v^{\alpha,s};\rho_\alpha^s)$ and $(\widetilde{v}^{\alpha,s};\widetilde{\rho}_\alpha^s)$ are both $\Gamma$-linearized coordinate systems, the uniqueness implies that, for each fixed $n\in\mathbb{Z}_{\geq 1}$, the collection $\{c^{\beta^\prime,t}_{\alpha_1,s_1;\cdots;\alpha_n,s_n}\}$ satisfies the assumption of Lemma~\ref{lem:trick}. In particular, we have
    \begin{equation*}
        \iota^*\left(\frac{\partial v^{\beta^\prime,t}}{\partial \widetilde{v}^{\alpha^\pprime,s}}\right)=\sum_{n\geq 0}\frac{1}{n !}\sum_{s_1,\ldots,s_n\geq 0} c^{\beta^\prime,t}_{\alpha^\pprime,s;\alpha_1^\prime,s_1;\cdots;\alpha_n^\prime,s_n}\widetilde{v}^{\alpha^\prime_1,s_1}\cdots \widetilde{v}^{\alpha^\prime_n,s_n}=0.
    \end{equation*}
    Similarly, one can show that
    \begin{equation*}
        \iota^*\left(\frac{\partial \rho^t_{\beta^{\prime}}}{\partial \widetilde{v}^{\alpha^\pprime,s}}\right)=0.
    \end{equation*}
    Therefore, we have
    \begin{align*}
        \iota^*\left(\frac{\partial Q}{\partial \widetilde{v}^{\alpha^\pprime,s}}\right)=&\sum_{t\geq 0}\iota^*\left(\frac{\partial Q}{\partial v^{\beta,t}}\right)\iota^*\left(\frac{\partial v^{\beta,t}}{\partial \widetilde{v}^{\alpha^\pprime,s}}\right)
        +\iota^*\left(\frac{\partial \rho^t_\beta}{\partial \widetilde{v}^{\alpha^\pprime,s}}\right)\iota^*\left(\frac{\partial Q}{\partial \rho^t_\beta}\right)\\
        =&\sum_{t\geq 0}\iota^*\left(\frac{\partial Q}{\partial v^{\beta^\prime,t}}\right)\iota^*\left(\frac{\partial v^{\beta^\prime,t}}{\partial \widetilde{v}^{\alpha^\pprime,s}}\right)+\iota^*\left(\frac{\partial Q}{\partial v^{\beta^\pprime,t}}\right)\iota^*\left(\frac{\partial v^{\beta^\pprime,t}}{\partial \widetilde{v}^{\alpha^\pprime,s}}\right)\\
        &\qquad+\iota^*\left(\frac{\partial \rho^t_{\beta^\prime}}{\partial \widetilde{v}^{\alpha^\pprime,s}}\right)\iota^*\left(\frac{\partial Q}{\partial \rho^t_{\beta^\prime}}\right)+\iota^*\left(\frac{\partial \rho^t_{\beta^\pprime}}{\partial \widetilde{v}^{\alpha^\pprime,s}}\right)\iota^*\left(\frac{\partial Q}{\partial \rho^t_{\beta^\pprime}}\right)\\
        =&0.
    \end{align*}
    The remaining half of the condition \eqref{eq:main} can be proved similarly.
\end{proof}

\section{CohFT and integrable hierarchies}
\label{sec:CohFT}
\subsection{The Moduli space of stable curves}
Let $\overline{\mathcal{M}}_{g,n}$ be the moduli space of genus $g$ stable curves with $n$ marked points. It is well-defined only in the stable range $2g-2+n>0$ and it is a smooth Deligne-Mumford stack over $\mathbb{C}$ of complex dimension $3g-3+n$. Define $\deg\colon H^*\left(\overline{\mathcal{M}}_{g,n},\mathbb{C}\right)\rightarrow H^*\left(\overline{\mathcal{M}}_{g,n},\mathbb{C}\right)$ to be the $\mathbb{C}$-linear operator which acts on $H^i\left(\overline{\mathcal{M}}_{g,n},\mathbb{C}\right)$ by multiplication by $\frac{i}{2}$.

The moduli space $\overline{\mathcal{M}}_{g,n}$ has a universal family given by forgetting the last marked point
\begin{equation}
    \pi\colon\overline{\mathcal{M}}_{g,n+1}\rightarrow\overline{\mathcal{M}}_{g,n}.
    \label{eq:unifam}
\end{equation}
This universal family has $n$ universal sections
\begin{equation*}
    \sigma_i\colon\overline{\mathcal{M}}_{g,n}\rightarrow\overline{\mathcal{M}}_{g,n+1}
\end{equation*}
given by attaching a $\mathbb{P}^1$-bubble to the $i$-th marked point and introducing two extra marked points labeled by $i$ and $n+1$ on the smooth locus of this $\mathbb{P}^1$-bubble. Let $\omega_\pi$ be the relative dualizing sheaf of the universal family \eqref{eq:unifam}. There are many tautologically constructed cohomology classes associated with it. Define the $\psi$-classes by
\begin{equation}
    \psi_i=c_1(\sigma_i^*(\omega_\pi))\in H^2\left(\overline{\mathcal{M}}_{g,n},\mathbb{C}\right),\quad1\leq i\leq n.
    \label{eq:psiclass}
\end{equation}
It is known that $\pi_*\omega_\pi$ is a locally free sheaf on $\overline{\mathcal{M}}_{g,n}$, called the Hodge bundle. Define the $\lambda$-classes by
\begin{equation*}
    \lambda_i=c_i(\pi_*\omega_\pi)\in H^{2i}\left(\overline{\mathcal{M}}_{g,n},\mathbb{C}\right).
\end{equation*}
Given $g_1, g_2, n_1, n_2\in\mathbb{Z}_{\geq 0}$ satisfying
\begin{equation*}
    2g_i-1+n_i>0,\,i=1,2.
\end{equation*}
Set $g=g_1+g_2,\, n=n_1+n_2$. There is a \textit{gluing tree} morphism
\begin{equation}
    \xi_{tree}\colon\overline{\mathcal{M}}_{g_1,n_1+1}\times\overline{\mathcal{M}}_{g_2,n_2+1}\rightarrow\overline{\mathcal{M}}_{g,n}
    \label{eq:gltree}
\end{equation}
defined by gluing the last marked point of the first stable curve to the first marked point of the second one. Given a positive integer $g$ and a nonnegative integer $n$ satisfying
\begin{equation*}
    2g-2+n>0,
\end{equation*}
There is a \textit{gluing loop} morphism
\begin{equation}
    \xi_{loop}\colon\overline{\mathcal{M}}_{g-1,n+2}\rightarrow\overline{\mathcal{M}}_{g,n}.
    \label{eq:glloop}
\end{equation}
given by gluing the last two marked points of a stable curve.

\subsection{Homogeneous cohomological field theory}
From now on, we fix a symmetric nondegenerate bilinear form $\eta\colon H\times H\rightarrow\mathbb{C}$, a $\mathbb{C}$-gradation $|\cdot|$ on $H$, and a distinguished element $\phi_1\in H\setminus\{0\}$ with $|\phi_1|=0$. We extend $\phi_1$ to a homogeneous basis $\{\phi_\alpha\}$. Let $\eta_{\alpha\beta}=\eta(\phi_\alpha,\phi_\beta)$ and $(\eta^{\alpha\beta})=(\eta_{\alpha\beta})^{-1}$.
\begin{defn}[\cite{KM}]
    A homogeneous cohomological field theory, or CohFT for short, is a collection of $\mathbb{C}$-linear maps
    \begin{equation}
        \Lambda_{g,n}\colon H^{\otimes n}\rightarrow H^*\left(\overline{\mathcal{M}}_{g,n},\mathbb{C}\right),\quad 2g-2+n>0
        \label{eq:cohftmap}
    \end{equation}
    satisfying the following properties:
    \begin{enumerate}[label=$(C\arabic*)$,ref=$(C\arabic*)$]
        \item \label{item:equi}(Equivariance) The map $\Lambda_{g,n}$ is equivariant under the symmetric group $\mathfrak{S}_n$. Here $\mathfrak{S}_n$ acts on $H^{\otimes n}$ by permuting copies of $H$ and on $H^*\left(\overline{\mathcal{M}}_{g,n},\mathbb{C}\right)$ by permuting marked points.
        \item \label{item:id}(Identity) For all $e_i\in H$, we have
        \begin{align*}
            &\Lambda_{0,3}(e_1\otimes e_2\otimes \phi_1)=\eta(e_1,e_2),\\
            &\Lambda_{g,n+1}(e_1\otimes\cdots\otimes e_n\otimes\phi_1)=\pi^*\Lambda_{g,n}(e_1\otimes\cdots\otimes e_n).
        \end{align*}
        Here we identify $H^*\left(\overline{\mathcal{M}}_{0,3},\mathbb{C}\right)$ with $\mathbb{C}$ via the canonical isomorphism $\overline{\mathcal{M}}_{0,3}\cong\Spec\mathbb{C}$.
        \item \label{item:gltree}(Gluing trees) Let $g=g_1+g_2,n=n_1+n_2$. For all $e_i\in H$, we have
        \begin{equation*}
            \xi_{tree}^*\Lambda_{g,n}(e_1\otimes\cdots\otimes e_n)=\Lambda_{g_1,n_1+1}(e_1,\otimes\cdots\otimes e_{n_1},\phi_\alpha)\otimes\eta^{\alpha\beta}\Lambda_{g_2,n_2+1}(\phi_\beta\otimes e_{n_1+1}\otimes\cdots\otimes e_{n}).
        \end{equation*}
        Here we use the Ku\"nneth isomorphism 
        \begin{equation*}
            H^*\left(\overline{\mathcal{M}}_{g_1,n_1+1}\times\overline{\mathcal{M}}_{g_2,n_2+1},\mathbb{C}\right)\cong H^*\left(\overline{\mathcal{M}}_{g_1,n_1+1},\mathbb{C}\right)\otimes H^*\left(\overline{\mathcal{M}}_{g_2,n_2+1},\mathbb{C}\right).
        \end{equation*}
        \item \label{item:glloop}(Gluing loops) We have
        \begin{equation*}
            \xi_{loop}^*\Lambda_{g,n}(e_1\otimes\cdots\otimes e_n)=\Lambda_{g-1,n+2}(e_1\otimes\cdots\otimes e_n\otimes\phi_\alpha\otimes\phi_\beta)\eta^{\alpha\beta}
        \end{equation*}
        for all $e_i\in H$.
        \item \label{item:homo}(Homogeneity) There exist $d,r^\alpha\in\mathbb{C}$ such that
        \begin{equation*}
            \pi_*\Lambda_{g,n+1}(e_1\otimes\cdots\otimes e_n\otimes r^\beta\phi_\beta)=\bigl(d(g-1)+\sum_{i=1}^n |e_i|-\deg\bigr)\Lambda_{g,n}(e_1\otimes\cdots\otimes e_n)
        \end{equation*}
        for all homogeneous $e_i\in H$.
    \end{enumerate}
\end{defn}

\begin{rmk}
    Comparing with the original definition given in \cite{KM}, we require additionally the axioms \ref{item:id}, \ref{item:homo}.
\end{rmk}

Given a CohFT $\{\Lambda_{g,n}\}$, we can define the \textit{quantum product} on $H$ by
\begin{equation*}
    \phi_\alpha*\phi_\beta=\eta^{\mu\nu}\Lambda_{0,3}(\phi_\alpha\otimes\phi_\beta\otimes\phi_\mu)\phi_\nu.
\end{equation*}
The axioms \ref{item:equi} and \ref{item:gltree} imply the commutativity and associativity of the quantum product respectively.

Let $H^{(p)}$ be the formal completion of $H$ at the origin. Whenever we choose a basis $\{\phi_\alpha\}$ of $H$, we have the following identifications of topological rings:
\begin{equation*}
    \mathcal{O}_{H^{(p)}}=\mathbb{C}[[t^{\alpha,p}\mid 1\leq\alpha\leq N]].
\end{equation*}
We also use the convention $t^\alpha=t^{\alpha,0}$. Define $\mathcal{O}_{H^\infty}$ to be the completion of the ring 
\begin{equation*}
    \mathbb{C}[[\varepsilon]][t^{\alpha,p}\mid 1\leq\alpha\leq N,\,p\geq 0]
\end{equation*}
with respect to the valuation \cite{CH}
\begin{equation}
    v(t^{\alpha,p})=p+1,\quad 1\leq\alpha\leq N,\,p\geq 0.
    \label{eq:valuation}
\end{equation}
As a ring, it is isomorphic to the formal power series ring in variables $\varepsilon,\,t^{\alpha,p}$:
\begin{equation*}
    \mathcal{O}_{H^\infty}=\mathbb{C}[[\varepsilon, t^{\alpha,p}\mid 1\leq\alpha\leq N,\,p\geq 0]].
\end{equation*}
Define the \textit{big phase space} $H^\infty$ to be the formal scheme whose underlying topological space is a single point and whose coordinate ring is $\mathcal{O}_{H^\infty}$.

Define the (ancestor) correlators
\begin{equation*}
    \langle\tau_{p_1}(e_1)\dots\tau_{p_n}(e_n)\rangle_{g}=\int_{\overline{\mathcal{M}}_{g,n}}\Lambda_{g,n}(e_1\otimes\cdots\otimes e_n)\psi_1^{p_1}\cdots\psi_n^{p_n},\quad e_i\in H,\,p_i\in\mathbb{Z}_{\geq 0}.
\end{equation*}
and the following generating series
\begin{equation*}
    \mathcal{F}=\sum_{g\geq 0}\varepsilon^{2g}\mathcal{F}_g=\sum_{\substack{g,n\geq 0\\2g-2+n>0}}\frac{\varepsilon^{2g}}{n!}\sum_{p_1,\ldots,p_n\in\mathbb{Z}_{\geq 0}}\langle\tau_{p_1}(\phi_{\alpha_1})\dots\tau_{p_n}(\phi_{\alpha_n})\rangle_{g}t^{\alpha_1,p_1}\cdots t^{\alpha_n,p_n}.
\end{equation*}
Here we adopt the convention in \cite{BDGR1} and write $\mathcal{F}=\sum_{g\geq 0}\varepsilon^{2g}\mathcal{F}_g$ instead of the standard convention $\mathcal{F}=\sum_{g\geq 0}\varepsilon^{2g-2}\mathcal{F}_g$. These series $\mathcal{F},\mathcal{F}_g$ are functions on the big phase space $H^\infty$. Consider the following series
\begin{equation*}
    F(t^1,\ldots,t^N)=\mathcal{F}_0\Big\vert_{t^{\bullet,>0}\rightarrow 0}\in \mathcal{O}_{H^{(0)}}.
\end{equation*}
It defines a (formal) Frobenius manifold with charge $d$ and Euler vector field
\begin{equation}
    E=\sum_{\alpha}\bigl((1-|\phi_\alpha|)t^\alpha+r^\alpha\bigr)\frac{\partial}{\partial t^\alpha}.
\end{equation}
For a detailed introduction to the theory of Frobenius manifolds, the readers may refer to \cite{Dub}.

Let $U$ be a sufficiently small neighborhood of $0\in H$ in the complex analytic topology and $s=s^\alpha\phi_\alpha\in U$. Since $H^*\left(\overline{\mathcal{M}}_{g,n},\mathbb{C}\right)$ is a finite dimensional $\mathbb{C}$-vector space, it can be identified with  $\mathbb{C}^{\dim_{\mathbb{C}}H^*\left(\overline{\mathcal{M}}_{g,n},\mathbb{C}\right)}$ as topological spaces. Suppose that
    \begin{equation}
        \Lambda^s_{g,n}(e_1\otimes\cdots\otimes e_n):=\sum_{m\geq 0}\frac{1}{m!}(\pi_m)_*\Lambda_{g,n+m}(e_1\otimes\cdots\otimes e_n\otimes s^{\alpha_1}\phi_{\alpha_1}\otimes\cdots\otimes s^{\alpha_m}\phi_{\alpha_m})
        \label{eq:shift}
    \end{equation}
is convergent in $U$. Here $\pi_m\colon\overline{\mathcal{M}}_{g,n+m}\rightarrow\overline{\mathcal{M}}_{g,n}$ is defined by forgetting the last $m$ marked points of a stable curve. These maps $\{\Lambda_{g,n}^s\}$ define a family of CohFT with $r^\alpha$ replaced by $r^\alpha+s^\alpha(1-|\phi_\alpha|)$ (see e.g. \cite{PPZ}).
\begin{defn}
    We call $s\in U$ a semisimple point of $\{\Lambda_{g,n}\}$ if the $\mathbb{C}$-algebra $H$, equipped with the quantum product associated with $\{\Lambda_{g,n}^s\}$, has no nonzero nilpotent elements, and say that $\{\Lambda_{g,n}\}$ is semisimple if the semisimple points form an open dense subset of $U$. 
\end{defn}

\subsection{CohFT with a finite symmetry}
In this section, we recall the notion of a finite symmetry of a CohFT. In the semisimple case, we provide a useful criterion for verifying finite symmetry.
\begin{defn}[\cite{LRZ}, Section 2.2]
    A CohFT is said to have finite symmetry $\Gamma$ if it is equipped with a $\Gamma$-action on $H$ satisfying
        \begin{enumerate}[label=$(S\arabic*)$,ref=$(S\arabic*)$]
        \item \label{item:finitesymmetry1} The $\Gamma$-action preserves the gradation and the distinguished element $\phi_1$,
        \item \label{item:finitesymmetry2} The maps $\{\Lambda_{g,n}\}$ are $\Gamma$-invariant. Here we extend the
        $\Gamma$-action to $H^{\otimes n}$ by
                \begin{equation*}
                    \gamma.(e_1\otimes\cdots\otimes e_n)\colon=\gamma.e_1\otimes\cdots\otimes\gamma.e_n.
                \end{equation*}
        \item \label{item:finitesymmetry3} We have
                \begin{equation}
                    r^{\alpha^\pprime}=0,\quad\forall\alpha^\pprime\in J.
                    \label{eq:deg}
                \end{equation}
            Here $r^\al$ are defined by the homogeneity axiom \ref{item:homo}.
    \end{enumerate}
    \label{def:finitesymmetry} 
\end{defn}
The following lemma is a consequence of the definition.
\begin{lem}
     Let $\{\Lambda_{g,n}\}$ be a CohFT with finite symmetry $\Gamma$. The metric $\eta$ is $\Gamma$-invariant, the decomposition \eqref{eq:decom} is orthogonal with respect to $\eta$, and moreover, for every $e_i\in H^\Gamma,\,e\in H^\mov$, we have
    \begin{equation*}
        \Lambda_{g,n+1}(e_1\otimes\cdots\otimes e_n\otimes e)=0.
    \end{equation*}
    \label{lem:equi2vanish}
\end{lem}

\begin{proof}
    Define
    \begin{equation*}
        c^{g,n}_{\alpha_1,s_1;\cdots;\alpha_n,s_n}=\Lambda_{g,n}(\phi_{\alpha_1}\otimes\cdots\otimes\phi_{\alpha_n}).
    \end{equation*}
    This lemma then follows from the cohomology-valued version of Lemma~\ref{lem:trick}, whose proof remains valid after replacing $\mathbb{C}$ by $H^*\left(\overline{\mathcal{M}}_{g,n},\mathbb{C}\right)$.
\end{proof}

For a semisimple CohFT, to prove that a $\Gamma$-action on $H$ defines a finite symmetry, it suffices to verify the above conditions for the genus-zero theory at a semisimple point.
\begin{lem}
    Let $\{\Lambda_{g,n}\}$ be a CohFT. Suppose that there is a $\Gamma$-action on $H$ satisfying \ref{item:finitesymmetry1}, \ref{item:finitesymmetry3}, and $\Lambda_{0,3}^s$ is $\Gamma$-invariant at a semisimple point $s=s^{\alpha^\prime}\phi_{\alpha^\prime}$. Then this $\Gamma$-action defines a finite symmetry.
\end{lem}

\begin{proof}
    For every $\gamma\in\Gamma$, define
    \begin{equation*}
        \Lambda_{g,n}^{s,\gamma}(e_1\otimes\cdots\otimes e_n):=\Lambda_{g,n}^s(\gamma.e_1\otimes\cdots\otimes\gamma.e_n),\quad\forall e_i\in H.
    \end{equation*}
    Note that $\{\Lambda_{g,n}^{s,\gamma}\}$ is a CohFT satisfying $\Lambda_{0,3}^{s,\gamma}=\Lambda_{0,3}^s$ and having the same homogeneity data in axiom \ref{item:homo} as $\{\Lambda_{g,n}^s\}$. The uniqueness part of the Givental-Teleman reconstruction theorem (see Theorem~1 in \cite{Tel}) forces
    \begin{equation*}
        \Lambda_{g,n}^s=\Lambda_{g,n}^{s,\gamma},\quad\forall g,n\geq 0,\,2g-2+n>0.
    \end{equation*}
    Shifting both sides by $-s$ as in \eqref{eq:shift}, we obtain the $\Gamma$-invariance of $\Lambda_{g,n}$.
\end{proof}

Let $\{\Lambda_{g,n}\}$ be a CohFT with finite symmetry $\Gamma$. We extend $\phi_1$ to a basis $\{\phi_{\alpha^\prime}\}_{\alpha^\prime\in I}$ of $H^\Gamma$ and choose a basis $\{\phi_{\alpha^\pprime}\}_{\alpha^\pprime\in J}$ for $H^\mov$. Let
\begin{equation*}
    \Lambda_{g,n}^\Gamma\colon\left(H^\Gamma\right)^{\otimes n}\rightarrow H^*\left(\overline{\mathcal{M}}_{g,n},\mathbb{C}\right)
\end{equation*}
be the restriction of $\Lambda_{g,n}$ to $\left(H^\Gamma\right)^{\otimes n}$. We can still define the spaces $H^{\Gamma,(p)}$, the big phase space $H^{\Gamma,\infty}$, its (ancestor) correlators
\begin{equation*}
    \langle\tau_{p_1}(e_1)\dots\tau_{p_n}(e_n)\rangle_{g}^\Gamma=\int_{\overline{\mathcal{M}}_{g,n}}\Lambda_{g,n}^\Gamma(e_1\otimes\ldots\otimes e_n)\psi_1^{p_1}\cdots\psi_n^{p_n},\quad e_i\in H^\Gamma,\, p_i\in\mathbb{Z}_{\geq 0},
\end{equation*}
and its generating series
\begin{equation*}
    \mathcal{F}^\Gamma=\sum_{g\geq 0}\varepsilon^{2g}\mathcal{F}_g^\Gamma=\sum_{\substack{g,n\geq 0\\2g-2+n>0}}\frac{\varepsilon^{2g}}{n!}\sum_{p_1,\ldots,p_n\in\mathbb{Z}_{\geq 0}}\langle\tau_{p_1}(\phi_{\alpha_1^\prime})\dots\tau_{p_n}(\phi_{\alpha_n^\prime})\rangle_{g}^\Gamma t^{\alpha_1^\prime,p_1}\cdots t^{\alpha_n^\prime,p_n}.
\end{equation*}
The inclusion $\iota\colon X^\Gamma\hookrightarrow X$ gives rise to a local homomorphism of $\mathbb{C}$-algebras
\begin{equation*}
    \iota^{(p),*}\colon\mathcal{O}_{H^{(p)}}\cong\mathbb{C}[[t^{\alpha,p}\mid 1\leq\alpha\leq N]]\rightarrow\mathcal{O}_{H^{\Gamma,(p)}}\cong\mathbb{C}[[t^{\alpha^\prime,p}\mid\alpha^\prime\in I]]
\end{equation*}
defined by taking quotient with respect to the variables indexed by $\alpha^\pprime\in J$, and therefore a local homomorphism $\iota^{\infty,*}\colon\mathcal{O}_{H^\infty}\rightarrow\mathcal{O}_{H^{\Gamma,\infty}}$ satisfies
\begin{equation}
    \mathcal{F}^\Gamma=\iota^{\infty,*}\mathcal{F},\quad\mathcal{F}_g^\Gamma=\iota^{\infty,*}\mathcal{F}_g.
\end{equation}
These homomorphisms are compatible with the induced $\Gamma$-actions on $\mathcal{O}_{H^{(p)}}$ and on $\mathcal{O}_{H^\infty}$, defined by
\begin{equation*}
    \gamma.t^{\alpha,p}=\act(\gamma^{-1})^\alpha_\beta t^{\beta,p},\quad\forall\gamma\in\Gamma.
\end{equation*}

\begin{defn}[\cite{LRZ}, Definition~2.7]
    A quadruple $(H,\eta,\phi_1\in H,\{\Lambda_{g,n}\})$ is called a partial CohFT if it satisfies \ref{item:equi},\ref{item:id},\ref{item:gltree}, and \ref{item:homo}.
\end{defn}
It follows from the definition of partial CohFTs that the genus-zero theory of a partial CohFT defines a Frobenius manifold.
\begin{lem}[\cite{LRZ}, Proposition~2.6]
    Let $(H,\eta,\phi_1\in H,\{\Lambda_{g,n}\})$ be a CohFT with finite symmetry $\Gamma$. Then the quadruple $(H^\Gamma,\eta\big\vert_{H^\Gamma\times H^\Gamma},\phi_1,\{\Lambda_{g,n}^\Gamma\})$ is a partial CohFT. In particular, $F^\Gamma:=\iota^{(0),*}F$ defines a Frobenius manifold.
\end{lem}
\begin{proof}
    It follows directly from lemma~\ref{lem:equi2vanish}.
\end{proof}

\subsection{DR hierarchy associated with a partial CohFT}
Let $a_1,\ldots,a_n\in\mathbb{Z}$ be a list of integers satisfying $\sum_{i=1}^n a_i=0$. The positive  (resp. absolute value of negative) integers among them define a partition $\mu^{+}$ (resp. $\mu^{-}$) of $\frac{1}{2}\sum_{i=1}^n |a_i|$. Let $n_0$ be the number of zeros among these $a_i$'s. Denote by $\overline{\mathcal{M}}_{g,n_0}^{\sim}(\mathbb{P}^1,\mu^{+},\mu^{-})$ the moduli space of stable maps to rubber with ramification profiles specified by $\mu^{\pm}$ (see \cite{LLZ,OP}). It has a virtual fundamental class
\begin{equation*}
    \left[\overline{\mathcal{M}}_{g,n_0}^{\sim}(\mathbb{P}^1,\mu^{+},\mu^{-})\right]^{\text{vir}}\in H_{2(2g-3+n)}\left(\overline{\mathcal{M}}_{g,n_0}^{\sim}(\mathbb{P}^1,\mu^{+},\mu^{-}),\mathbb{C}\right).
\end{equation*}
There is a morphism given by forgetting the maps
\begin{equation*}
    \st\colon\overline{\mathcal{M}}_{g,n_0}^{\sim}(\mathbb{P}^1,\mu^{+},\mu^{-})\rightarrow\overline{\mathcal{M}}_{g,n}.
\end{equation*}
Define the double ramification (DR) cycle by
\begin{equation*}
    \DR_g(a_1,\ldots,a_n)=\st_*\left[\overline{\mathcal{M}}_{g,n_0}^{\sim}(\mathbb{P}^1,\mu^{+},\mu^{-})\right]^{\text{vir}}\in H^{2g}\left(\overline{\mathcal{M}}_{g,n},\mathbb{C}\right).
\end{equation*}
Here we use the Poincar\'e duality to identify $H^{2g}(\overline{\mathcal{M}}_{g,n},\mathbb{C})$ and $H_{2(2g-3+n)}(\overline{\mathcal{M}}_{g,n},\mathbb{C})$. It is shown in \cite{JPPZ} that
\begin{equation*}
    \DR_g(a_1,\ldots,a_n)\in RH^{2g}\left(\overline{\mathcal{M}}_{g,n}\right).
\end{equation*}
To construct the double ramification hierarchy, we need a key Lemma:
\begin{lem}[\cite{Bur}, Lemma 3.2]
    For every class $A\in H^*\left(\overline{\mathcal{M}}_{g,n};\mathbb{C}\right)$, the integral
    \begin{equation}
        \int_{\overline{\mathcal{M}}_{g,n}}\DR_g(a_1,\ldots,a_n)\lambda_g A
        \label{eq:drpoly}
    \end{equation}
    is a polynomial in $a_1,\ldots,a_n$ of degree $2g$.
\end{lem}

Recall that we use the notation $(u^{\alpha,s};\theta_\alpha^s)$ to denote the dual coordinate system (See Definition~\ref{defn:dual}). Given a partial CohFT $\{\Lambda_{g,n}\}$, we can define a series of differential polynomials following \cite{BR1}:
\begin{equation}
    \begin{aligned}
        g_{\alpha,p}=&\sum_{\substack{g,n\geq 0\\2g-1+n>0}}\frac{\varepsilon^{2g}}{n!}\sum_{\substack{s_1,\ldots,s_n\in\mathbb{Z}_{\geq 0}\\s_1+\cdots+s_n=2g}}u^{\alpha_1,s_1}\cdots u^{\alpha_n,s_n}\\
        &\times\Coef_{a_1^{s_1}\cdots a_n^{s_n}}\int_{\overline{\mathcal{M}}_{g,n+1}}\DR_g(-\sum_{i=1}^n a_i,a_1,\ldots,a_n)\lambda_g\psi_1^p\Lambda_{g,n+1}(\phi_\alpha\otimes\phi_{\alpha_1}\otimes\cdots\otimes\phi_{\alpha_n}),
    \end{aligned}
    \label{eq:drhamiltonian}
\end{equation}
for all $1\leq\alpha\leq N$ and all $p\geq 0$. Denote $\overline{g}_{\alpha,p}=\int g_{\alpha,p}\in\hat{\mathscr{F}}^0$. 
\begin{defn}[\cite{Bur,BDGR1}]
    The double ramification (DR) hierarchy associated with a partial CohFT $\{\Lambda_{g,n}\}$ is defined to be the following system
    \begin{equation}
        \frac{\partial u^\alpha}{\partial t^{\beta,p}}=\eta^{\alpha\mu}\partial\frac{\delta\overline{g}_{\beta,p}}{\delta u^\mu},\quad 1\leq\alpha,\beta\leq N,\, p\geq 0.
        \label{eq:dr}
    \end{equation}
\end{defn}

The matrix differential operator 
\begin{equation*}
    \mathscr{P}_1^{\DR,\alpha\beta}=\eta^{\alpha\beta}\partial
\end{equation*}
satisfies the antisymmetry property \eqref{eq:antisym} and the leading term property \eqref{eq:leadingterm}. Therefore,  it corresponds to a local functional
\begin{equation}
    P_1^{\DR}=\int\frac{1}{2}\eta^{\alpha\beta}\theta_{\alpha}\theta_{\beta}^1\in\hat{\mathscr{F}}^2_1
    \label{eq:drhamiltonianstructure1}
\end{equation}
and defines a bracket $\{-,-\}_{P_1^\DR}$ on $\hat{\mathscr{F}}^0$. It is easy to check that $P_1^{DR}$ is a Hamiltonian structure. 

Define the following differential polynomial
\begin{align*}
    g=&\sum_{\substack{g,n\geq 0\\2g-2+n>0}}\frac{\varepsilon^{2g}}{n!}\sum_{\substack{s_1,\ldots,s_n\in\mathbb{Z}_{\geq 0}\\s_1+\cdots+s_n=2g}}u^{\alpha_1,s_1}\cdots u^{\alpha_n,s_n}\\
    &\times\Coef_{a_1^{s_1}\cdots a_n^{s_n}}\int_{\overline{\mathcal{M}}_{g,n}}\DR_g(a_1,\ldots,a_n)\lambda_g\Lambda_{g,n}(\phi_{\alpha_1}\otimes\cdots\otimes\phi_{\alpha_n}).
\end{align*}
Throughout this paper, we use the notation $g$ for both this differential polynomial and genera, in accordance with the convention of \cite{BR2}. The meaning of g will always be clear from the context.
Denote $\overline{g}=\int g$. Note that the integral \eqref{eq:drpoly} is defined only when $\sum_{i=1}^n a_i=0$, and therefore its polynomial representation is not unique. For example, two such representations may differ by a polynomial in $\sum_{i=1}^n a_i$. In particular, the differential polynomial $g$ depends on the choice of polynomial representations of the integrals involved in its definition. However, the local functional $\overline{g}$ is uniquely determined (see \cite{Bur,BR1}).

In \cite{BRS}, the authors define the following matrix differential operators
\begin{equation}
    \begin{aligned}
        \mathscr{P}_2^{\DR,\alpha\beta}=&\sum_{s\geq 0}\Bigg((1+d+s-|\phi_\alpha|-|\phi_\beta|)\eta^{\alpha\mu}\eta^{\beta\nu}\frac{\partial}{\partial u^{\nu,s}}\left(\frac{\delta\overline{g}}{\delta u^\mu}\right)\partial^{s+1}\\
        &+(\frac{1}{2}+\frac{d}{2}+s-|\phi_\beta|)\eta^{\alpha\mu}\eta^{\beta\nu}\partial\left(\frac{\partial}{\partial u^{\nu,s}}\left(\frac{\delta\overline{g}}{\delta u^\mu}\right)\right)\partial^s\Bigg)\\
        &+\eta^{\alpha\mu}\eta^{\beta\nu}\Lambda_{0,3}(\phi_\mu\otimes\phi_\nu\otimes r^\zeta\phi_\zeta)\partial.
    \end{aligned}
    \label{eq:drhamiltonianstructure2}
\end{equation}
It also satisfies \eqref{eq:antisym} and \eqref{eq:leadingterm}, and therefore corresponds to a local functional $P_2^{\DR}\in\hat{\mathscr{F}}^2_{\geq 1}$. It is shown in loc. cit. that the equality $[P_1^\DR,P_2^\DR]=0$ holds for a CohFT. Moreover, their proof of this identity is almost entirely combinatorial. The only aspect of CohFT they use is the identity $(|\phi_\alpha|+|\phi_\beta|-d)\eta^{\alpha\beta}=0$, which remains valid for a partial CohFT. Therefore, the same equality remains valid. It is shown in \cite{BR2} that for a semisimple CohFT, the equality $[P_2^\DR,P_2^\DR]=0$ also holds.

\section{$\Gamma$-reduction of hierarchies and DR/DZ equivalence}
\label{sec:red}
Let $\{\Lambda_{g,n}\}$ be a semisimple CohFT with finite symmetry $\Gamma$. We extend $\phi_1$ to a basis $\{\phi_{\alpha^\prime}\}_{\alpha^\prime\in I}$ of $H^\Gamma$ and choose a basis $\{\phi_{\alpha^\pprime}\}_{\alpha^\pprime\in J}$ for $H^\mov$. Now we have two DR hierarchies associated with $\{\Lambda_{g,n}\}$ and $\{\Lambda_{g,n}^\Gamma\}$ respectively. We call the former the ambient DR hierarchy and the latter the reduced DR hierarchy.

\subsection{Reduction of the DR hierarchy and its bihamiltonian structure}
In \cite{BDGR1}, the authors define the double ramification correlators for the CohFT $\{\Lambda_{g,n}\}$
\begin{equation*}
    \langle\tau_{p_1}(e_1)\cdots\tau_{p_n}(e_n)\rangle_g^\DR\in\mathbb{C},\quad e_i\in H,\,g,n\geq 0,\,2g-2+n>0.
\end{equation*}
They depend linearly on each $e_i$, and are invariant under permutations of pairs $\{(p_i,e_i)\}$. There is a geometric formula for DR correlators (see \cite{BDGR1}, Propositions~6.1 and~6.10;
\cite{BDGR2}, Theorem~6.1). In particular, this formula implies that DR correlators satisfy
\begin{equation*}
    \langle\tau_{p_1}(\gamma.e_1)\cdots\tau_{p_n}(\gamma.e_n)\rangle_g^\DR=\langle\tau_{p_1}(e_1)\cdots\tau_{p_n}(e_n)\rangle_g^\DR,\quad\forall\gamma\in\Gamma.
\end{equation*}
By lemma~\ref{lem:trick}, it follows that
\begin{equation*}
    \langle\tau_{p_0}(e)\tau_{p_1}(e_1)\cdots\tau_{p_n}(e_n)\rangle_g^\DR=0,\quad\forall e_i\in H^\Gamma,\,e\in H^\mov.
\end{equation*}
We can also define the generating series of DR correlators
\begin{equation*}
    \mathcal{F}^\DR=\sum_{g\geq 0}\varepsilon^{2g}\mathcal{F}^\DR_g=\sum_{\substack{g,n\geq 0\\2g-2+n>0}}\frac{\varepsilon^{2g}}{n!}\sum_{p_1,\ldots,p_n\in\mathbb{Z}_{\geq 0}}\langle\tau_{p_1}(\phi_{\alpha_1})\dots\tau_{p_n}(\phi_{\alpha_n})\rangle_{g}^\DR t^{\alpha_1,p_1}\cdots t^{\alpha_n,p_n}.
\end{equation*}
It is invariant under the $\Gamma$-action on $\mathcal{O}_{H^\infty}$.

\begin{lem}[\cite{BDGR1}, Proposition 6.3]
    The DR generating series satisfies the string equation
    \begin{equation*}
        \frac{\partial\mathcal{F}^\DR}{\partial t^{1,0}}=\sum_{p\geq 0}t^{\alpha,p+1}\frac{\partial\mathcal{F}^\DR}{\partial t^{\alpha,p}}+\frac{1}{2}\eta_{\alpha\beta}t^{\alpha,0}t^{\beta,0}.
    \end{equation*}
\end{lem}

Consider the following functions on the big phase space $H^\infty$
\begin{equation}
    \widetilde{u}^{\str,\alpha,s}=\left(\frac{\partial}{\partial t^{1,0}}\right)^s\left(\eta^{\alpha\beta}\frac{\partial^2\mathcal{F}^\DR}{\partial t^{1,0}\partial t^{\beta,0}}\right)\in\mathcal{O}_{H^\infty}.
    \label{eq:drnormalcoor}
\end{equation}
Define $\mathcal{O}_{H^\infty}^{\geq s}$ to be the ideal of $\mathcal{O}_{H^\infty}$ consisting of formal power series all of whose monomials $\prod_{i=1}^n t^{\alpha_i,p_i}$ satisfies $\sum_{i=1}^n p_i\geq s$. Applying the string equation inductively, one can show that
\begin{equation*}
    (\widetilde{u}^{\str,\alpha,s}-\delta^\alpha_1\delta^s_1)-t^{\alpha,s}\in\mathcal{O}_{H^\infty}^{\geq s+1}.
\end{equation*}
Define $\mathscr{A}^{\wk}_{\widetilde{u}}$ to be the completion of the ring
\begin{equation*}
    \mathbb{C}[[\varepsilon]][\widetilde{u}^{\alpha,s}-\delta_1^\alpha\delta_1^s]
\end{equation*}
with respect to the valuation
\begin{equation*}
    v(\widetilde{u}^{\alpha,s}-\delta_1^\alpha\delta_1^s)=s+1,\quad 1\leq\alpha\leq N,\,s\geq 0.
\end{equation*}
As a ring, it is isomorphic to the formal power series ring in variables $\varepsilon,\,\widetilde{u}^{\alpha,s}-\delta_1^\alpha\delta_1^s$:
\begin{equation}
    \mathscr{A}^\wk_{\widetilde{u}}=\mathbb{C}[[\varepsilon]][[\widetilde{u}^{\alpha,s}-\delta_1^\alpha\delta_1^s]].
    \label{eq:weak}
\end{equation}
We have the following inclusion map
\begin{equation*}
    \hat{\mathscr{A}}_{\widetilde{u}}^0\hookrightarrow\mathscr{A}^{\wk}_{\widetilde{u}},\quad\varepsilon^s\widetilde{u}^{\alpha,s}\mapsto \varepsilon^s(\widetilde{u}^{\alpha,s}-\delta^\alpha_1\delta^s_1)+\delta^\alpha_1\delta^s_1\varepsilon^s.
\end{equation*}
The induced coordinate system $(\widetilde{u}^{\alpha,s};\widetilde{\theta}_\alpha^s)$ on $J^\infty(\hat{\mathfrak{X}})$ is called the DR normal coordinate system. Now, the functions \eqref{eq:drnormalcoor} define an isomorphism of rings
\begin{equation}
    \mathscr{A}^{\wk}_{\widetilde{u}}\rightarrow\mathcal{O}_{H^\infty},\quad \widetilde{u}^{\alpha,s}-\delta^\alpha_1\delta^s_1\mapsto \widetilde{u}^{\str,\alpha,s}-\delta^\alpha_1\delta^s_1.
    \label{eq:drnormal2bigphase}
\end{equation}
There is a unique element $\Omega^\DR_{\alpha,p;\beta,q}\in\hat{\mathscr{A}}^0_{\widetilde{u}}\subset\mathscr{A}^{\wk}_{\widetilde{u}}$ such that
\begin{equation*}
    \Omega^\DR_{\alpha,p;\beta,q}\Big\vert_{\widetilde{u}^{\bullet,\bullet}\rightarrow\widetilde{u}^{\str,\bullet,\bullet}}=\frac{\partial^2\mathcal{F}^\DR}{\partial t^{\alpha,p}\partial t^{\beta,q}}
\end{equation*}
for any $1\leq\alpha,\beta\leq N$ and $p,q\geq 0$ (see \cite{BDGR1}). 

In the same paper, the authors also introduce a Miura-type transformation
\begin{equation}
    u^\alpha\mapsto \widetilde{u}^\alpha=\eta^{\alpha\mu}\frac{\delta\overline{g}_{\mu,0}}{\delta u^1}.
    \label{eq:dr2normal}
\end{equation}
of the 2nd kind and show that the ambient DR hierarchy \eqref{eq:dr} in the DR normal coordinate system becomes
\begin{equation}
    \frac{\partial \widetilde{u}^\alpha}{\partial t^{\beta,p}}=\eta^{\alpha\mu}\partial\Omega^\DR_{\mu,0;\beta,p},\quad 1\leq\alpha,\beta\leq N,\, p\geq 0,
    \label{eq:normaldr}
\end{equation}
and
\begin{equation*}
    \widetilde{u}^{\alpha}:=\widetilde{u}^{\str,\alpha}\Big\vert_{t^{1,0}\rightarrow t^{1,0}+x}
\end{equation*}
is a solution of it. Here the action of the global vector field $\partial$ should be understood as taking derivatives with respect to the formal loop variable $x$. 

By the explicit formula \eqref{eq:drhamiltonian} of $g_{\mu,0}$, one can easily check that this Miura-type transformation is $\Gamma$-linearized and therefore, the DR normal coordinate is a $\Gamma$-linearized coordinate system by Lemma~\ref{lem:linearizedmiura}. We can extend the $\Gamma$-action on $\hat{\mathscr{A}}^0_{\widetilde{u}}$ to $\mathscr{A}^{\wk}_{\widetilde{u}}$ in an obvious way. This makes the isomorphism \eqref{eq:drnormal2bigphase} $\Gamma$-equivariant.

In summary, we have the following commutative diagram
\begin{equation*}
    \begin{tikzcd}
        \hat{\mathscr{A}}_{u} \arrow[d,"\iota^*"]  & \hat{\mathscr{A}}_{\widetilde{u}} \arrow[l, "\cong","\eqref{eq:dr2normal}"'] \arrow[d,"\iota^*"] & \hat{\mathscr{A}}_{\widetilde{u}}^0 \arrow[l,hook'] \arrow[r,hook] \arrow[d,"\iota^*"] & \mathscr{A}^{\wk}_{\widetilde{u}} \arrow[r,"\cong"',"\eqref{eq:drnormal2bigphase}"] \arrow[d,"\iota^*"] & \mathcal{O}_{H^\infty} \arrow[d,"\iota^{\infty,*}"]\\
        \hat{\mathscr{A}}^\Gamma_{u}  & \hat{\mathscr{A}}^\Gamma_{\widetilde{u}} \arrow[l, "\cong"] & \hat{\mathscr{A}}_{\widetilde{u}}^{\Gamma,0} \arrow[l,hook'] \arrow[r,hook] & \mathscr{A}^{\Gamma,\wk}_{\widetilde{u}} \arrow[r,"\cong"'] & \mathcal{O}_{H^{\Gamma,\infty}}
    \end{tikzcd}.
\end{equation*}
Here the horizontal arrows in the top rows are $\Gamma$-equivariant, the bottom horizontal arrows are constructed analogously for $\{\Lambda_{g,n}^\Gamma\}$, and the vertical arrows are all given by quotienting out variables with indices in $J$. In particular, the restriction of the CohFT $\{\Lambda_{g,n}\}$ to $\{\Lambda_{g,n}^{\Gamma}\}$ corresponds to the reduction of the associated DR hierarchies to the $\Gamma$-invariant flows.

Next, we show that the bihamiltonian structure also admits a reduction.
\begin{lem}
    For all $i=1,2,\,\mu^\prime\in I$, and $p\geq 0$, the following equations hold:
    \begin{align*}
        \overline{\iota}^*[-,P^\DR_i]&=[\overline{\iota^*}(-),\overline{\iota}^*(P_i^\DR)]^\Gamma,\\
        \overline{\iota}^*[-,\overline{g}_{\mu^\prime,p}]&=[\overline{\iota^*}(-),\overline{\iota}^*(\overline{g}_{\mu^\prime,p})]^\Gamma,\\
        \overline{\iota}^*[-,[P^\DR_i,\overline{g}_{\mu^\prime,p}]]&=[\overline{\iota^*}(-),\overline{\iota}^*([P^\DR_i,\overline{g}_{\mu^\prime,p}])]^\Gamma.
    \end{align*}
    \label{lem:poissonpreserve}
\end{lem}

\begin{proof}
    The proofs of these equations are the same, so we only prove the first one here.
    According to Lemma~\ref{lem:main}, it suffices to show that the equation \eqref{eq:main} holds for certain representatives of $P^\DR_i,\,i=1,2$. We will check for $\widetilde{P}^\DR_i=\frac{1}{2}\theta_\alpha \mathscr{P}_i^{\DR,\alpha\beta}(\theta_\beta)$. The verification for $i=1$ is trivial, so we assume $i=2$. Note that the conditions under consideration are linear. Hence, we can check it term by term. 
    
    According to Lemma~\ref{lem:equi2vanish}, we have 
    \begin{align*}
        &\eta^{\alpha\mu}\eta^{\beta\nu}\Lambda_{0,3}(\phi_\mu,\phi_\nu,r^\zeta\phi_\zeta)\theta_\alpha\theta_\beta^1
        =\eta^{\alpha\mu}\eta^{\beta\nu}\Lambda_{0,3}(\phi_\mu,\phi_\nu,r^{\zeta^\prime}\phi_{\zeta^\prime})\theta_\alpha\theta_\beta^1\\
        =&\,\eta^{\alpha^\prime\mu^\prime}\eta^{\beta^\prime\nu^\prime}\Lambda_{0,3}(\phi_{\mu^\prime},\phi_{\nu^\prime},r^{\zeta^\prime}\phi_{\zeta^\prime})\theta_{\alpha^\prime}\theta_{\beta^\prime}^1
        +\eta^{\alpha^\pprime\mu^\pprime}\eta^{\beta^\pprime\nu^\pprime}\Lambda_{0,3}(\phi_{\mu^\pprime},\phi_{\nu^\pprime},r^{\zeta^\prime}\phi_{\zeta^\prime})\theta_{\alpha^\pprime}\theta_{\beta^\pprime}^1,
    \end{align*}
    and the resulting expression satisfies \eqref{eq:main}. Denote
    \begin{equation*}
        \widetilde{\mathscr{P}}_2^{\DR,\alpha\beta}=\mathscr{P}_2^{\DR,\alpha\beta}-\eta^{\alpha\mu}\eta^{\beta\nu}\Lambda_{0,3}(\phi_\mu,\phi_\nu,r^\zeta\phi_\zeta)\partial.
    \end{equation*}
    For $\alpha^\prime,\,\beta^\prime\in I$, the coefficients of the differential operator $\widetilde{\mathscr{P}}_2^{\DR,\alpha^\prime\beta^\prime}$ are sums of terms of the following form: 
    \begin{align*}
        &\eta^{\alpha^\prime\mu^\prime}\eta^{\beta^\prime\nu^\prime}(\text{Constant})\times u^{\alpha_1,s_1}\cdots u^{\alpha_n,s_n}\\
        &\times\Coef_{\ldots}\int_{\overline{\mathcal{M}}_{g,n+2}}(\text{Tautological classes})\times\Lambda_{g,n+2}(\phi_{\mu^\prime},\phi_{\nu^\prime},\phi_{\alpha_1},\ldots,\phi_{\alpha_n}).
    \end{align*}
    No summation over the indices $\alpha_1,\ldots,\alpha_n$ is assumed here. The notation $\Coef_{\ldots}$ means that the integral behind it is a polynomial of certain parameter and we take a coefficient of it. By Lemma~\ref{lem:equi2vanish}, if exactly one index among $\alpha_1,\ldots,\alpha_n$ lies in $J$, then the integral vanishes. In particular, a nonzero such term must satisfy \eqref{eq:main}. It follows that the same condition holds for
    \begin{equation*}
        \frac{1}{2}\theta_{\alpha^\prime}\widetilde{\mathscr{P}}_i^{\DR,\alpha^\prime\beta^\prime}(\theta_{\beta^\prime}).
    \end{equation*}
    The same argument applies to
    \begin{equation*}
        \frac{1}{2}\theta_{\alpha^\prime}\widetilde{\mathscr{P}}_i^{\DR,\alpha^\prime\beta^\pprime}(\theta_{\beta^\pprime}),\quad\frac{1}{2}\theta_{\alpha^\pprime}\widetilde{\mathscr{P}}_i^{\DR,\alpha^\pprime\beta^\prime}(\theta_{\beta^\prime}).
    \end{equation*}
    As for the term
    \begin{equation*}
        \frac{1}{2}\theta_{\alpha^\pprime}\widetilde{\mathscr{P}}_i^{\DR,\alpha^\pprime\beta^\pprime}(\theta_{\beta^\pprime}),
    \end{equation*}
    it depends quadratically in odd variables $\theta_{\alpha^\pprime}^s$ with index $\alpha^\pprime$ in $J$.
\end{proof}

Let $(P_1^{\DR,\Gamma},P_2^{\DR,\Gamma})$ be the pair of local functionals and $g_{\alpha^\prime,p}^\Gamma\in\hat{\mathscr{A}}^0$ be the differential polynomials associated with the reduced DR hierarchy. These objects are defined by \eqref{eq:drhamiltonianstructure1}, \eqref{eq:drhamiltonianstructure2}, and \eqref{eq:drhamiltonian} respectively.

\begin{thm}[Reduction of the bihamiltonian structure]For the reduced DR hierarchy, we have the following results:
    \begin{enumerate}[label=$\arabic*.$,ref=$\arabic*$]
        \item The pair $(P^{\DR,\Gamma}_1,P^{\DR,\Gamma}_2)$ is a bihamiltonian structure.
        \item The following bihamiltonian recursion relation holds on $\hat{\mathscr{F}}^{\Gamma,0}$:
        \begin{align*}
            \{-,\overline{g}^\Gamma_{\alpha^\prime,p}\}^\Gamma_{P^{\DR,\Gamma}_2}=&\left(p+\frac{3-d}{2}+|\phi_{\alpha^\prime}|\right)\{-,\overline{g}^\Gamma_{\alpha^\prime,p+1}\}^\Gamma_{P^{\DR,\Gamma}_1}\\
            &+\Lambda_{0,3}(\eta^{\beta^\prime\mu^\prime}\phi_{\mu^\prime},\phi_{\alpha^\prime},r^{\zeta^\prime}\phi_{\zeta^\prime})\{-,\overline{g}^\Gamma_{\beta^\prime,p}\}^\Gamma_{P^{\DR,\Gamma}_1},\quad\alpha^\prime\in I.
        \end{align*}
    \end{enumerate}
    In particular, the reduced DR hierarchy admits a bihamiltonian structure.
    \label{thm:main}
\end{thm}

Our proof of Theorem~\ref{thm:main} relies on the following result, which says that the ambient DR hierarchy is bihamiltonian.
\begin{thm}[\cite{BR2}, Theorem 3.4]
    For the ambient DR hierarchy, we have
    \begin{enumerate}[label=$\arabic*.$,ref=$\arabic*$]
        \item The pair $(P^{\DR}_1,P^{\DR}_2)$ is a bihamiltonian structure.
        \item The following bihamiltonian recursion relation holds on $\hat{\mathscr{F}}^{0}$:
        \begin{equation*}
            \{-,\overline{g}_{\alpha,p}\}_{P^{\DR}_2}=\left(p+\frac{3-d}{2}+|\phi_{\alpha}|\right)\{-,\overline{g}_{\alpha,p+1}\}_{P^{\DR}_1}+\Lambda_{0,3}(\eta^{\beta\mu}\phi_{\mu},\phi_{\alpha},r^{\zeta}\phi_{\zeta})\{-,\overline{g}_{\beta,p}\}_{P^{\DR}_1}.
        \end{equation*}
    \end{enumerate}
    \label{thm:ssbiham}
\end{thm}

\begin{proof}[Proof of Theorem~\ref{thm:main}]
    By definition, we have
    \begin{align*}
        &\overline{\iota}^*(P^\DR_i)=P^{\DR,\Gamma}_i,\quad i=1,2,\\
        &\overline{g}^\Gamma_{\alpha^\prime,p}=\overline{\iota}^*(\overline{g}_{\alpha^\prime,p}),\quad\alpha^\prime\in I,\,p\geq 0.
    \end{align*}
By Lemma~\ref{lem:poissonpreserve} and Theorem~\ref{thm:ssbiham}, we have
    \begin{equation*}
        [P_i^{\DR,\Gamma},P_j^{\DR,\Gamma}]^\Gamma=\overline{\iota}^*[P_i^\DR,P_j^\DR]=0,\quad\forall 1\leq i,j\leq 2.
    \end{equation*}
    The bihamiltonian recursion relation in Theorem~\ref{thm:ssbiham} is equivalent to
    \begin{align*}
        [-,[P^\DR_2,\overline{g}_{\alpha,p}]]=&\left(p+\frac{3-d}{2}+|\phi_{\alpha}|\right)[-,[P^\DR_1,\overline{g}_{\alpha,p+1}]]\\
        &+\Lambda_{0,3}(\eta^{\beta\mu}\phi_{\mu},\phi_{\alpha},r^{\zeta}\phi_{\zeta})[-,[P^\DR_1,\overline{g}_{\beta,p}]],
    \end{align*}
    on $\hat{\mathscr{F}}^0$. Applying $\overline{\iota}^*$ to both sides, we obtain the following identity on $\hat{\mathscr{F}}^0$:
    \begin{align*}
        &[\overline{\iota}^*(-),[P^{\DR,\Gamma}_2,\overline{g}^\Gamma_{\alpha^\prime,p}]^\Gamma]^\Gamma\\
        =&\,[\overline{\iota}^*(-),[\overline{\iota}^*(P^\DR_2),\overline{\iota}^*(\overline{g}_{\alpha^\prime,p})]^\Gamma]^\Gamma
        =\overline{\iota}^*\left([(-),[P_2^\DR,\overline{g}_{\alpha^\prime,p}]]\right)\\
        =&\left(p+\frac{3-d}{2}+|\phi_{\alpha^\prime}|\right)[\overline{\iota}^*(-),[P^{\DR,\Gamma}_1,\overline{g}^\Gamma_{\alpha^\prime,p+1}]^\Gamma]^\Gamma\\
        &\quad+\Lambda_{0,3}(\eta^{\beta\mu}\phi_{\mu},\phi_{\alpha^\prime},r^{\zeta^\prime}\phi_{\zeta^\prime
        })[\overline{\iota}^*(-),[P^{\DR,\Gamma}_1,\overline{\iota}^*(\overline{g}_{\beta,p})]^\Gamma]^\Gamma\\
        =&\left(p+\frac{3-d}{2}+|\phi_{\alpha^\prime}|\right)[\overline{\iota}^*(-),[P^{\DR,\Gamma}_1,\overline{g}^\Gamma_{\alpha^\prime,p+1}]^\Gamma]^\Gamma\\
        &\quad+\Lambda_{0,3}(\eta^{\beta^\prime\mu^\prime}\phi_{\mu^\prime},\phi_{\alpha^\prime},r^{\zeta^\prime}\phi_{\zeta^\prime
        })[\overline{\iota}^*(-),[P^{\DR,\Gamma}_1,\overline{g}^\Gamma_{\beta^\prime,p}]^\Gamma]^\Gamma,\quad\alpha^\prime\in I.
    \end{align*}
    We conclude by noting that $\overline{\iota}^*$ is surjective.
\end{proof}

\subsection{Reduction of DZ hierarchy and DR/DZ equivalence}
The DZ hierarchy can be constructed from the given semisimple CohFT $\{\Lambda_{g,n}\}$ as follows \cite{BPS1,BPS2,DZ3}. Consider the following functions on the big phase space $H^\infty$
\begin{equation}
    w^{\tp,\alpha,s}=\left(\frac{\partial}{\partial t^{1,0}}\right)^s\left(\eta^{\alpha\beta}\frac{\partial^2\mathcal{F}}{\partial t^{1,0}\partial t^{\beta,0}}\right)\in\mathcal{O}_{H^\infty}.
    \label{eq:dznormalcoor}
\end{equation}
Recall that $\mathcal{F}$ satisfies the string equation
\begin{equation}
    \frac{\partial\mathcal{F}}{\partial t^{1,0}}=\sum_{p\geq 0}t^{\alpha,p+1}\frac{\partial\mathcal{F}}{\partial t^{\alpha,p}}+\frac{1}{2}\eta_{\alpha\beta}t^{\alpha,0}t^{\beta,0}.
    \label{eq:stringeqn}
\end{equation}
This implies that
\begin{equation*}
    (w^{\tp,\alpha,s}-\delta^\alpha_1\delta^s_1)-t^{\alpha,s}\in\mathcal{O}_{H^\infty}^{\geq s+1}.    
\end{equation*}
Define $\mathscr{A}^{\wk}_w$ in the same way as in \eqref{eq:weak}. It contains $\mathbb{\hat{\mathscr{A}}}^0_w$ as a subalgebra. The induced coordinate system $(w^{\alpha,s};\sigma_\alpha^s)$ is called the DZ normal coordinate system. We can also define the following isomorphism of rings
\begin{equation}
    \mathscr{A}^{\wk}_w\rightarrow\mathcal{O}_{H^\infty},\quad w^{\alpha,s}-\delta^\alpha_1\delta^s_1\mapsto w^{\tp,\alpha,s}-\delta^\alpha_1\delta^s_1.
    \label{eq:dznormal2bigphase}
\end{equation}
There is a unique element $\Omega_{\alpha,p;\beta,q}\in\hat{\mathscr{A}}_w^0\subset\mathscr{A}^{\wk}_w$ such that
\begin{equation*}
    \Omega^\DZ_{\alpha,p;\beta,q}\Big\vert_{w^{\bullet,\bullet}\rightarrow w^{\tp,\bullet,\bullet}}=\frac{\partial^2\mathcal{F}}{\partial t^{\alpha,p}\partial t^{\beta,q}}
\end{equation*}
for all $1\leq\alpha,\beta\leq N$ and $p,q\geq 0$ (see \cite{BPS1}).

We define a $\Gamma$-action on $\hat{{\mathscr{A}}}_w$ by requiring that the DZ normal coordinate system is a linearized coordinate system, i.e.
\begin{equation*}
    \gamma._w w^{\alpha,s}=\act(\gamma^{-1})^\alpha_\beta w^{\beta,s},\quad\gamma._w\sigma_\alpha^s=\act(\gamma)_\alpha^\beta\sigma_\beta^s.
\end{equation*}
To distinguish it from the previous $\Gamma$-action, we will call it the $\Gamma_w$-action and use the notation $._w$ for this action. Define $\iota^*_w$ to be the map setting all variables with indices in $J$ to zero, in analogy with the map $\iota^*$ in a $\Gamma$-linearized coordinate system. The $\Gamma_w$-action on $\hat{\mathscr{A}}_w^0$ can be extended naturally to $\mathscr{A}^{\wk}_w$. One can check that the $\Gamma_w$-action makes the isomorphism $\eqref{eq:dznormal2bigphase}$ $\Gamma_w$-equivariant. 

\begin{defn}
    The Dubrovin-Zhang (DZ) hierarchy associated with the semisimple CohFT $\{\Lambda_{g,n}\}$ is defined by
    \begin{equation}
        \frac{\partial w^\alpha}{\partial t^{\beta,p}}=\eta^{\alpha\mu}\partial\Omega^\DZ_{\mu,0;\beta,p},\quad 1\leq\alpha,\beta\leq N,\, p\geq 0.
        \label{eq:dz}
    \end{equation}
\end{defn}
There is a solution to the DZ hierarchy given by:
\begin{equation*}
    w^{\alpha}:=w^{\tp,\alpha}\Big\vert_{t^{1,0}\rightarrow t^{1,0}+x}.
\end{equation*}

Although the DZ hierarchy is defined only for semisimple CohFTs, we can still define an analogue for the partial CohFT $\{\Lambda_{g,n}^\Gamma\}$.
\begin{defn}
    The DZ hierarchy associated with $\{\Lambda^\Gamma_{g,n}\}$ is defined by
    \begin{equation}
        \frac{\partial w^{\alpha^\prime}}{\partial t^{\beta^\prime,p}}=\eta^{\alpha^\prime\mu^\prime}\partial\Omega^{\Gamma,\DZ}_{\mu^\prime,0;\beta^\prime,p},\quad \alpha^\prime,\beta^\prime\in I,\, p\geq 0.
        \label{eq:reduceddz}
    \end{equation}
    where $\Omega^{\Gamma,\DZ}_{\mu^\prime,0;\beta^\prime,p}=\iota^*_w\Omega^{\DZ}_{\mu^\prime,0;\beta^\prime,p}$
\end{defn}

As before, we call the DZ hierarchy associated with $\{\Lambda_{g,n}\}$ the ambient hierarchy and the one associated with $\{\Lambda^\Gamma_{g,n}\}$ the reduced hierarchy.

\begin{lem}
    The flows of the reduced DZ hierarchy \eqref{eq:reduceddz} are commutative.
\end{lem}

\begin{proof}
    Note that each flow $\frac{\partial}{\partial t^{\beta,p}}$ of the ambient DZ hierarchy defines a derivation $D_{\beta,p}$ on $\hat{\mathscr{A}}^0_w$ via 
    \begin{equation*}
        \frac{\partial}{\partial t^{\beta,p}}\circ\partial=\partial\circ\frac{\partial}{\partial t^{\beta,p}}.
    \end{equation*}
    The commutativity of the ambient DZ hierarchy implies that the derivations $D_{\beta,p}$ are mutually commuting. The ideal $\hat{\mathscr{I}}_w\cap\hat{\mathscr{A}}^0$ of $\hat{\mathscr{A}}^0$ is generated by $w^{\alpha^\pprime,s},\,\alpha^\pprime\in J$, where $\hat{\mathscr{I}}_w$ is the kernel of $\iota^*_w$. By Lemma~\ref{lem:equi2vanish}, we have
    \begin{align*}
        &\iota^*_w\left(D_{\beta^\prime,p}(w^{\alpha^\pprime,s})\right)\Big\vert_{w^{\bullet,\bullet}\rightarrow w^{\tp,\bullet,\bullet}}\\
        =&\,\iota^{\infty,*}\left(\eta^{\alpha^\pprime\mu^\pprime}\left(\frac{\partial}{\partial t^{1,0}}\right)^{s+1}\frac{\partial^2 \mathcal{F}}{\partial t^{\beta^\prime,p}\partial t^{\mu^\pprime,0}}\right)\\
        =&\,\eta^{\alpha^\pprime\mu^\pprime}\sum_{g,n\geq 0}\frac{\varepsilon^{2g}}{n !}\sum_{p_1,\ldots,p_n\geq 0}\langle\left(\tau_0(\phi_1)\right)^{s+1}\tau_{p}(\phi_{\beta^\prime})\tau_0(\phi_{\mu^\pprime})\tau_{p_1}(\phi_{\alpha_1^\prime})\cdots\tau_{p_n}(\phi_{\alpha_n^\prime})\rangle t^{\alpha^\prime_1,p_1}\cdots t^{\alpha^\prime_n,p_n}\\
        =&\,0,\qquad\forall \alpha^\pprime\in J,\,\beta^\prime\in I,\, s,p\geq 0.
    \end{align*}
    Applying the isomorphism \eqref{eq:dznormal2bigphase}, we have $D_{\beta^\prime,p}(\hat{\mathscr{I}}_w\cap\hat{\mathscr{A}}^0)\subset\hat{\mathscr{I}}_w\cap\hat{\mathscr{A}}^0,\,\forall\beta^\prime\in I$. So each derivation $D_{\beta^\prime,p}$ descends to a derivation $D_{\beta^\prime,p}^\Gamma$ on $\hat{\mathscr{A}}^{\Gamma,0}_w$. Since the map $\iota^*_w$ is surjective, the operators $D^\Gamma_{\beta^\prime,p}$ are determined by
    \begin{equation*}
        \iota^*_w\circ D_{\beta^\prime,p}=D^\Gamma_{\beta^\prime,p}\circ\iota_w^*,\quad\beta^\prime\in I,\,p\geq 0.
    \end{equation*}
    The fact that the operators $\{D_{\beta^\prime,p}\}$ commute implies that the operators $\{D_{\beta^\prime,p}^\Gamma\}$ commute with each other. It is straightforward to check that the derivations defined by the flows of the reduced DZ hierarchy \eqref{eq:reduceddz} coincide with $D_{\beta^\prime,p}^\Gamma$.
\end{proof}

It has recently been proved that for a semisimple CohFT, the DR hierarchy and the DZ hierarchy are related by a Miura-type transformation of the 2nd kind.
\begin{thm}[\cite{BLS}, Theorem 1.2]
    There is a Miura-type transformation of the 2nd kind $u^\alpha\mapsto w^\alpha$ transforming the ambient DR hierarchy to the ambient DZ hierarchy.
    \label{thm:drdz}
\end{thm}

\begin{rmk}
    The ambient DZ hierarchy admits a bihamiltonian structure \cite{DZ3,LWZ2}, which can be constructed directly. In \cite{BDGR1,BR2,BRS}, it was proved that the Miura-type transformation in Theorem~\ref{thm:drdz} preserves the bihamiltonian structures. Therefore, one can alternatively define the bihamiltonian structures for \eqref{eq:dz} by
    \begin{equation*}
        P_i^\DZ:=P_i^\DR\bigl(u(w)\bigr).
    \end{equation*}
\end{rmk}

This Miura-type transformation can be described explicitly. A priori, the function
\begin{equation*}
    \mathcal{F}^\DR-\mathcal{F}\in\mathcal{O}_{H^\infty}
\end{equation*}
corresponds to a unique element in $\mathscr{A}^{\wk}_w$ via the inverse of \eqref{eq:dznormal2bigphase}. From Proposition 7.2, Proposition 7.4 in \cite{BDGR1} and Theorem~\ref{thm:drdz}, it follows that there is a unique differential polynomial $\mathcal{P}\in\hat{\mathscr{A}}^0_w$ such that
\begin{equation}
    \mathcal{F}^\DR-\mathcal{F}=\mathcal{P}\Big\vert_{w^{\bullet,\bullet}\rightarrow w^{\tp,\bullet,\bullet}}.
    \label{eq:defofp}
\end{equation}
Now the Miura-type transformation in Theorem~\ref{thm:drdz} is the composition of \eqref{eq:dr2normal} and the inverse of
\begin{equation}
    w^{\alpha}\mapsto\widetilde{u}^\alpha=w^\alpha+\sum_{s\geq 0}\eta^{\alpha\mu}\eta^{\beta\nu}\partial\left(\frac{\partial\mathcal P}{\partial w^{\beta,s}}\partial^{s+1}\Omega^\DZ_{\nu,0;\mu,0}\right).
    \label{eq:dz2drnormal}
\end{equation}

\begin{thm}
    The Miura-type transformation \eqref{eq:dz2drnormal} is $\Gamma_w$-linearized. In particular, the $\Gamma_w$-action coincides with the original $\Gamma$-action on $\hat{\mathscr{A}}$, and we have $\iota^*_w=\iota^*$.
\end{thm}

\begin{proof}
    For an element $\gamma\in\Gamma$, we write $\gamma._w(-)$ for the $\Gamma_w$-action by $\gamma\in\Gamma$. To check the equation \eqref{eq:linmiura} in $\hat{\mathscr{A}}_w^0$, we may apply \eqref{eq:dznormal2bigphase} and verify it in $\mathcal{O}_{H^\infty}$ instead. Since $\mathcal{P}$ is invariant under the $\Gamma_w$-action by \eqref{eq:defofp}, it follows that
    \begin{equation*}
        \gamma._w\left(\frac{\partial\mathcal{P}}{\partial w^{\beta,s}}\right)=\act(\gamma)^{\xi}_\beta\left(\frac{\partial\mathcal{P}}{\partial w^{\xi,s}}\right),\quad\forall 1\leq\beta\leq N,\, s\geq 0.
    \end{equation*}
    This yields that
    \begin{align*}
        &\gamma._w\left(\sum_{s\geq 0}\eta^{\alpha\mu}\eta^{\beta\nu}\partial\left(\frac{\partial\mathcal P}{\partial w^{\beta,s}}\partial^{s+1}\Omega^\DZ_{\nu,0;\mu,0}\right)\right)\Big\vert_{w^{\bullet,\bullet}\rightarrow w^{\tp,\bullet,\bullet}}\\
        =&\,\gamma.\left(\sum_{s\geq 0}\eta^{\alpha\mu}\eta^{\beta\nu}\partial\left(\frac{\partial\mathcal P}{\partial w^{\beta,s}}\partial^{s+1}\Omega^\DZ_{\nu,0;\mu,0}\right)\Big\vert_{w^{\bullet,\bullet}\rightarrow w^{\tp,\bullet,\bullet}}\right)\\
        =&\,\gamma.\left(\sum_{s\geq 0}\eta^{\alpha\mu}\eta^{\beta\nu}\Bigl(\frac{\partial}{\partial t^{1,0}}\Bigr)\left(\Bigl(\frac{\partial\mathcal P}{\partial w^{\beta,s}}\Bigr)\Big\vert_{w^{\bullet,\bullet}\rightarrow w^{\tp,\bullet,\bullet}}\times\Bigl(\frac{\partial}{\partial t^{1,0}}\Bigr)^{s+1}\Bigl(\frac{\partial^2\mathcal{F}}{\partial t^{\nu,0}\partial t^{\mu,0}}\Bigr)\right)\right)\\
        =&\sum_{s\geq 0}\eta^{\alpha\mu}\eta^{\beta\nu}\left(\frac{\partial}{\partial t^{1,0}}\right)\left(\act(\gamma)_\beta^\xi\Bigl(\frac{\partial\mathcal P}{\partial w^{\xi,s}}\Bigr)\Big\vert_{w^{\bullet,\bullet}\rightarrow w^{\tp,\bullet,\bullet}}\times\Bigl(\frac{\partial}{\partial t^{1,0}}\Bigr)^{s+1}\act(\gamma)_\nu^\delta\act(\gamma)_\mu^\lambda\Bigl(\frac{\partial^2\mathcal{F}}{\partial t^{\delta,0}\partial t^{\lambda,0}}\Bigr)\right)\\
        =&\sum_{s\geq 0}\left(\eta^{\alpha\mu}\act(\gamma)_\mu^\lambda\right)\left(\eta^{\beta\nu}\act(\gamma)_\beta^\xi\act(\gamma)_\nu^\delta\right)\\
        &\qquad\times\Bigl(\frac{\partial}{\partial t^{1,0}}\Bigr)\left(\Bigl(\frac{\partial\mathcal P}{\partial w^{\xi,s}}\Bigr)\Big\vert_{w^{\bullet,\bullet}\rightarrow w^{\tp,\bullet,\bullet}}\times\Bigl(\frac{\partial}{\partial t^{1,0}}\Bigr)^{s+1}\Bigl(\frac{\partial^2\mathcal{F}}{\partial t^{\delta,0}\partial t^{\lambda,0}}\Bigr)\right)\\
        =&\sum_{s\geq 0}\act(\gamma^{-1})^\alpha_\zeta\eta^{\zeta\lambda}\eta^{\xi\delta}\times\Bigl(\frac{\partial}{\partial t^{1,0}}\Bigr)\left(\Bigl(\frac{\partial\mathcal P}{\partial w^{\xi,s}}\Bigr)\Big\vert_{w^{\bullet,\bullet}\rightarrow w^{\tp,\bullet,\bullet}}\times\Bigl(\frac{\partial}{\partial t^{1,0}}\Bigr)^{s+1}\Bigl(\frac{\partial^2\mathcal{F}}{\partial t^{\delta,0}\partial t^{\lambda,0}}\Bigr)\right)\\
        =&\,\act(\gamma^{-1})^\alpha_{\zeta}\left(\sum_{s\geq 0}\eta^{\zeta\lambda}\eta^{\xi\delta}\Bigl(\frac{\partial}{\partial t^{1,0}}\Bigr)\left(\Bigl(\frac{\partial\mathcal P}{\partial w^{\xi,s}}\Bigr)\Big\vert_{w^{\bullet,\bullet}\rightarrow w^{\tp,\bullet,\bullet}}\times\Bigl(\frac{\partial}{\partial t^{1,0}}\Bigr)^{s+1}\Bigl(\frac{\partial^2\mathcal{F}}{\partial t^{\delta,0}\partial t^{\lambda,0}}\Bigr)\right)\right)\\
        =&\,\act(\gamma^{-1})^\alpha_{\zeta}\left(\sum_{s\geq 0}\eta^{\zeta\lambda}\eta^{\xi\delta}\partial\left(\frac{\partial\mathcal P}{\partial w^{\xi,s}}\partial^{s+1}\Omega^\DZ_{\delta,0;\lambda,0}\right)\right)\Big\vert_{w^{\bullet,\bullet}\mapsto w^{\tp,\bullet,\bullet}}.
    \end{align*}
    This shows that the Miura-type transformation \eqref{eq:dz2drnormal} is $\Gamma_w$-linearized. By Lemma~\ref{lem:linearizedmiura}, the DZ normal coordinate is a $\Gamma_w$-linearized coordinate system. This shows that the $\Gamma$-action and the $\Gamma_w$-action on the DZ normal coordinate system coincide.
\end{proof}

By Remark~\ref{rmk:inducedmiura}, we have the following DR/DZ correspondence for the reduced hierarchies.
\begin{thm}
    \label{thm:reduceddrdz}
    Let $u^\alpha\mapsto w^\alpha=u^\alpha+\sum_{d\geq 1}f^\alpha_d$  be the Miura-type transformation in Theorem~\ref{thm:drdz}. Then $u^{\alpha^\prime}\mapsto w^{\alpha^\prime}=u^{\alpha^\prime}+\sum_{d\geq 1}\iota^*(f^{\alpha^\prime}_d),\,\alpha^\prime\in I$ is a Miura-type transformation transforming the reduced DR hierarchy into the reduced DZ hierarchy. Moreover, this Miura-type transformation transforms $P_i^{\DR,\Gamma}$ into $P_i^{\DZ,\Gamma}=\overline{\iota}^*\left(P^\DZ_i\right)$, and endows the reduced DZ hierarchy with a bihamiltonian structure.
\end{thm}

In summary, we have the following commutative diagram of $\Gamma$-equivariant maps:
\begin{equation*}
    \begin{tikzcd}
        \hat{\mathscr{A}}_{w} \arrow[d,"\eqref{eq:dz2drnormal}","\cong"'] & \hat{\mathscr{A}}_{w}^0 \arrow[l,hook'] \arrow[r,hook] \arrow[d,"\cong"']& \mathscr{A}^{\wk}_w \arrow[dr,bend left, "\cong"', "\eqref{eq:dznormal2bigphase}"] \arrow[d,"\cong"']&\\
        \hat{\mathscr{A}}_{\widetilde{u}} & \hat{\mathscr{A}}_{\widetilde{u}}^0 \arrow[l,hook'] \arrow[r,hook] & \mathscr{A}^{\wk}_{\widetilde{u}} \arrow[r,"\cong"',"\eqref{eq:drnormal2bigphase}"] & \mathcal{O}_{H^\infty}
    \end{tikzcd}.
\end{equation*}
The arrow $\mathscr{A}^{\wk}_w\rightarrow\mathscr{A}^{\wk}_{\widetilde{u}}$ is induced from the Miura-type transformation \eqref{eq:dz2drnormal}.

\subsection{Reduction of tau structures}
Let $(v^{\alpha,s},\rho_\alpha^s)$ be a coordinate system, $P\in\hat{\mathscr{F}}^2_{v,\geq 1}$ be a Hamiltonian structure with $\varepsilon$-expansion
\begin{equation*}
    P=\sum_{m\geq 0}\varepsilon^{m+1}P^{[m]},\quad P^{[m]}\in\hat{\mathscr{F}}^2_{v,m+1},
\end{equation*}
and $\{\overline{h}_{\alpha,p}\in\hat{\mathscr{F}}^0\mid p\geq 0\}$ be a collection of local functionals satisfying
\begin{equation*}
    \{\overline{h}_{\alpha,p},\overline{h}_{\beta,q}\}_P=0.
\end{equation*}
Denote the matrix differential operator associated with $P^{[m]}$ by $\mathscr{P}^{[m]}=\left(\mathscr{P}^{\alpha\beta,[m]}\right)$. In particular, we have
\begin{equation*}
    \quad\mathscr{P}^{\alpha\beta,[0]}=g^{\alpha\beta}\partial+b^{\alpha\beta}_\mu v^{\mu,1},\quad g^{\alpha\beta},b^{\alpha\beta}_\mu\in\hat{\mathscr{A}}^0_{0}.
\end{equation*}
We require that
\begin{equation*}
    \det (g^{\alpha\beta})\Big\vert_{v^\bullet=0}\neq 0
\end{equation*}
and that
\begin{equation}
    \mathscr{P}^{\alpha\mu}\frac{\delta\overline{h}_{1,0}}{\delta v^{\mu}}=v^{\alpha,1},
    \label{eq:t10equalx}
\end{equation}
where $\mathscr{P}^{\alpha\beta}=\sum_{m\geq 0}\varepsilon^{m+1}\mathscr{P}^{\alpha\beta,[m]}$. Define the following Hamiltonian PDEs with commutative flows:
\begin{equation}
    \frac{\partial v^\alpha}{\partial t^{\beta,q}}=\mathscr{P^{\alpha\mu}}\frac{\delta\overline{h}_{\beta,q}}{\delta v^{\mu}},\quad\forall 1\leq \alpha,\beta\leq N,\,q\geq 0.
    \label{eq:tmphierarchy}
\end{equation}

\begin{defn}[\cite{BDGR1,DLZ3,DZ3}]
    A choice of differential polynomials $\{h_{\alpha,p}\in\hat{\mathscr{A}}^0_v\mid p\geq -1\}$ is called a collection of tau-symmetric Hamiltonian densities if
    \begin{enumerate}[label=$(D\arabic*)$,ref=$(D\arabic*)$]
        \item \label{item:casimir}The differential polynomials with $p=-1$ provide $N$ linearly independent Casimirs:
        \begin{equation*}
            \mathscr{P}^{\alpha\beta}\frac{\delta\overline{h}_{\mu,-1}}{\delta v^{\beta}}=0,\quad\overline{h}_{\alpha,-1}=\int h_{\alpha,-1}.
        \end{equation*}
        \item \label{item:representatives}The differential polynomials with $p\geq 0$ are representatives of $\overline{h}_{\alpha,p}$:
        \begin{equation*}
            \overline{h}_{\alpha,p}=\int h_{\alpha,p},\quad p\geq 0.
        \end{equation*}
        \item \label{item:tausymmetric}They satisfy the following tau-symmetric property:
        \begin{equation*}
            \frac{\partial h_{\alpha,p-1}}{\partial t^{\beta,q}}=\frac{\partial h_{\beta,q-1}}{\partial t^{\alpha,p}},\quad p,q\geq 0.
        \end{equation*}
    \end{enumerate}
\end{defn}

Given a collection of tau-symmetric Hamiltonian densities $\{h_{\alpha,p}\}$, there are differential polynomials $\Omega_{\alpha,p;\beta,q}\in\hat{\mathscr{A}}^0$ satisfying
\begin{equation*}
    \frac{\partial h_{\alpha,p-1}}{\partial t^{\beta,q}}=\partial\Omega_{\alpha,p;\beta,q},
\end{equation*}
since
\begin{equation*}
    \int\frac{\partial h_{\alpha,p-1}}{\partial t^{\beta,q}}=\{\overline{h}_{\alpha,p-1},\overline{h}_{\beta,q}\}_P=0,\quad p,q\geq 0,
\end{equation*}
by \eqref{eq:tmphierarchy}, \ref{item:casimir}, and \ref{item:representatives}. The differential polynomials $\Omega_{\alpha,p;\beta,q}$ are determined up to additive constants. Hence one can adjust these constants so that they satisfy the following additional properties.

\begin{defn}[\cite{BDGR1,DLZ3,DZ3}]
    A tau structure for a collection of tau-symmetric Hamiltonian densities $\{h_{\alpha,p}\}$ is a collection of differential polynomials $\{\Omega_{\alpha,p;\beta,q}\mid 1\leq\alpha,\beta\leq N,\,p,q\geq 0\}$ satisfying
    \begin{enumerate}[label=$(T\arabic*)$,ref=$(T\arabic*)$]
        \item $\partial\Omega_{\alpha,p;\beta,q}=\frac{\partial h_{\alpha,p-1}}{\partial t^{\beta,q}}$,
        \item $\Omega_{\alpha,p;1,0}=h_{\alpha,p-1}$,
        \item $\Omega_{\alpha,p;\beta,q}=\Omega_{\beta,q;\alpha,p}$,
    \end{enumerate}
    for all $1\leq\alpha,\beta\leq N,\,p,q\geq 0$.
    \label{def:ts}
\end{defn}

Given a tau structure $\{\Omega_{\alpha,p;\beta,q}\}$, the property \ref{item:tausymmetric} implies that the differential polynomial $\partial\left(\frac{\partial\Omega_{\alpha,p;\beta,q}}{\partial t^{\mu,r}}\right)$ is symmetric under the permutation of the pairs $(\alpha,p),\,(\beta,q),\,(\mu,r)$. Since the right-hand side of \eqref{eq:tmphierarchy} lies in $\hat{\mathscr{A}}^0_{\geq 1}$, the differential polynomial $\frac{\partial\Omega_{\alpha,p;\beta,q}}{\partial t^{\mu,r}}$ has no constant term. Therefore, the differential polynomial $\frac{\partial\Omega_{\alpha,p;\beta,q}}{\partial t^{\mu,r}}$ must itself be symmetric under the permutation of the pairs $(\alpha,p),\,(\beta,q),\,(\mu,r)$.

This symmetric property implies that given a formal power series solution $v^{\alpha}=\mathfrak{v}^\alpha(x,t^{\bullet,\bullet};\varepsilon)$ of the system \eqref{eq:tmphierarchy}, there is a formal power series $\mathfrak{F}$ in $t^{\alpha,p}$ such that
\begin{equation*}
    \frac{\partial^2\mathfrak{F}}{\partial t^{\alpha,p}\partial t^{\beta,q}}=\Omega_{\alpha,p;\beta,q}\Big\vert_{v^{\mu,s}\rightarrow \frac{\partial^s \mathfrak{v}^\mu}{\partial x^s}}.
\end{equation*}
The function $\mathfrak{F}$ is determined up to an affine linear polynomial in the variables $t^{\alpha,p}$. We call any such function $\mathfrak{F}$ a \textit{free energy} associated with this tau structure and $\exp(\varepsilon^{-2}\mathfrak{F})$ a tau function \cite{DLZ3,DZ3}.

The ambient (resp. reduced) DR hierarchy, together with the Hamiltonian structure $P_1^{\DR}$ $\left(\text{resp.\ } P_1^{\DR,\Gamma}\right)$, admits a canonical choice of tau-symmetric Hamiltonian densities \cite{BDGR1}:
\begin{equation*}
        h^\DR_{\alpha,p}:=\frac{\delta \overline{g}_{\alpha,p+1}}{\delta u^1}\ \left(\text{resp.\ } h^{\DR,\Gamma}_{\alpha^\prime,p}=\iota^*h^\DR_{\alpha^\prime,p}=\frac{\delta \overline{g}^\Gamma_{\alpha^\prime,p+1}}{\delta u^1}\right),
\end{equation*}
There is also a corresponding tau structure given by
\begin{equation*}
    \{\Omega^\DR_{\alpha,p;\beta,q}\}\ \left(\text{resp.\ } \{\Omega^{\DR,\Gamma}_{\alpha^\prime,p;\beta^\prime,q}:=\iota^*\Omega^\DR_{\alpha^\prime,p;\beta^\prime,q}\}\right),
\end{equation*}
and the generating function of the DR correlators $\mathcal{F}^\DR$ $\left(\text{resp.\ } \mathcal{F}^{\DR,\Gamma}=\iota^{\infty,*}\mathcal{F}^\DR\right)$ is a corresponding free energy for the solution 
\begin{equation*}
    u^{\str,\alpha}=u^\alpha(\widetilde{u})\Big\vert_{\widetilde{u}^{\bullet,\bullet}\rightarrow\widetilde{u}^{\str,\bullet,\bullet}\vert_{t^{1,0}\rightarrow t^{1,0}+x}}\ \left(\text{resp.\ } \iota^{\infty,*}u^{\str,\alpha^\prime}\right).
\end{equation*}

Associated with the ambient DZ hierarchy and its Hamiltonian structure $P_1^{\DZ}$, there is a tau structure 
\begin{equation*}
    \{\Omega^\DZ_{\alpha,p;\beta,q}\},
\end{equation*}
and the differential polynomials
\begin{equation*}
    h^\DZ_{\alpha,p}:=\Omega_{\alpha,p+1;1,0}^\DZ.
\end{equation*}
are tau-symmetric Hamiltonian densities \cite{DZ3}. The generating function $\mathcal{F}$ of the ancestor correlators is a corresponding free energy for the solution
\begin{equation*}
    w^{\tp,\alpha}\Big\vert_{t^{1,0}\rightarrow t^{1,0}+x}.
\end{equation*}

\begin{thm}
    The reduced DZ hierarchy together with the Hamiltonian structure $P_1^{\DZ,\Gamma}$ admits a collection of tau-symmetric Hamiltonian densities
    \begin{equation*}
         h^{\DZ,\Gamma}_{\alpha^\prime,p-1}=\iota^*h^\DZ_{\alpha^\prime,p-1},\quad\forall\alpha^\prime\in I,\,p\geq 0,
    \end{equation*}
    whose tau structure is given by
    \begin{equation*}
        \{\Omega^{\DZ,\Gamma}_{\alpha^\prime,p;\beta^\prime,q}\mid\forall\alpha^\prime,\beta^\prime\in I,\, p,q\geq 0\},
    \end{equation*}
    and the function $\mathcal{F}^\Gamma$ is a free energy corresponding to the solution
    \begin{equation*}
        \iota^{\infty,*}(w^{\tp,\alpha^\prime})\Big\vert_{t^{1,0}\rightarrow t^{1,0}+x}.
    \end{equation*}
    Moreover, the Miura-type transformation in Theorem~\ref{thm:reduceddrdz} preserves the tau-symmetric Hamiltonian densities and the tau structures in the sense that
    \begin{align*}
        &h_{\alpha^\prime,p-1}^{\DR,\Gamma}(u^\prime(w^\prime))=h_{\alpha^\prime,p-1}^{\DZ,\Gamma}+\partial D^\Gamma_{\alpha^\prime,p}\iota^*\mathcal{P},\\
        &\Omega_{\alpha^\prime,p;\beta^\prime,q}^{\DR,\Gamma}(u^\prime(w^\prime))=\Omega_{\alpha^\prime,p;\beta^\prime,q}^{\DZ,\Gamma}+D^\Gamma_{\alpha^\prime,p}D^\Gamma_{\beta^\prime,q}\iota^*\mathcal{P},
    \end{align*}
    where $D_{\alpha^\prime,p}^\Gamma$ are the differential operators defined by the flows of the reduced DZ hierarchy.
    \label{thm:taustructure}
\end{thm}

\begin{proof}
    Denote by $\mathscr{P}_1^{\DZ}$ (resp. $\mathscr{P}_1^{\DZ,\Gamma}$) the matrix differential operator associated with $P_1^\DZ$ (resp. $P_1^{\DZ,\Gamma}$). The conditions \ref{item:representatives} hold trivially. The differential polynomials $h^{\DZ,\Gamma}_{\alpha^\prime,-1}$ are linearly independent since we have $h^\DZ_{\alpha,-1}=\eta_{\alpha\beta}w^\beta$. 
    
    Note that the proof of Lemma~\ref{lem:poissonpreserve} shows that the differential polynomials $P_1^\DR$ and $g_{\alpha^\prime,p},\,\alpha^\prime\in I$ satisfy \eqref{eq:main}, and so do $h^{\DR}_{\alpha^\prime,p}$. It follows that the differential polynomials $P_1^\DZ$ also satisfy \eqref{eq:main} by Lemma~\ref{lem:trickapply}. Let $D_{\alpha,p}$ be the differential operators defined by the flows of the ambient DZ hierarchy. The discussion in \cite{BDGR1}, Section 3.6 together with Theorem~\ref{thm:drdz} shows that
    \begin{equation}
        \begin{aligned}
            &h_{\alpha,p-1}^{\DR}(u(w))=h_{\alpha,p-1}^{\DZ}+\partial D_{\alpha,p}\mathcal{P},\\
            &\Omega_{\alpha,p;\beta,q}^{\DR}(u(w))=\Omega_{\alpha,p;\beta,q}^{\DZ}+D_{\alpha,p}D_{\beta,q}\mathcal{P}
        \end{aligned}
        \label{eq:preservetau}
    \end{equation}
    which implies that the differential polynomials $h^{\DZ}_{\alpha^\prime,p},\,\alpha^\prime\in I$ also satisfy \eqref{eq:main}. Applying the trick in Lemma~\ref{lem:trickapply}, we can show by Lemma~\ref{lem:trick} that
    \begin{equation*}
        \iota^*\left(\frac{\delta h^\DZ_{\alpha^\prime,p}}{\delta w^{\beta^\pprime}}\right)=0,\quad\forall\beta^\pprime\in J.
    \end{equation*}
    It follows that the condition \ref{item:tausymmetric} holds:
    \begin{align*}
        \frac{\partial h^{\DZ,\Gamma}_{\alpha^\prime,p-1}}{\partial t^{\beta^\prime,q}}&=\sum_{s\geq 0}\frac{\partial h^{\DZ,\Gamma}_{\alpha^\prime,p-1}}{\partial w^{\mu^\prime,s}}\partial^s\left(\mathscr{P}_1^{\DZ,\Gamma,\mu^\prime\nu^\prime}\frac{\delta\overline{h}^{\DZ,\Gamma}_{\beta^\prime,q}}{\delta w^{\nu^\prime}}\right)=\iota^*\left(\sum_{s\geq 0}\frac{\partial h^{\DZ}_{\alpha^\prime,p-1}}{\partial w^{\mu,s}}\partial^s\left(\mathscr{P}_1^{\DZ,\mu\nu}\frac{\delta\overline{h}^\DZ_{\beta^\prime,q}}{\delta w^{\nu}}\right)\right)\\
        &=\iota^*\left(\frac{\partial h^{\DZ}_{\alpha^\prime,p-1}}{\partial t^{\beta^\prime,q}}\right)=\iota^*\left( \frac{\partial h^{\DZ}_{\beta^\prime,q-1}}{\partial t^{\alpha^\prime,p}}\right)=\frac{\partial h^{\DZ,\Gamma}_{\beta^\prime,q-1}}{\partial t^{\alpha^\prime,p}}.
    \end{align*}
    The second half of the condition \ref{item:casimir} can be verified analogously. The conditions in Definition~\ref{def:ts} hold for $\{\Omega_{\alpha^\prime,p;\beta^\prime,q}^{\DZ,\Gamma}\}$ since we have
    \begin{equation*}
        \iota^*\circ\frac{\partial}{\partial t^{\alpha^\prime,p}}=\frac{\partial}{\partial t^{\alpha^{\prime,p}}}\circ\iota^*.
    \end{equation*}
    The function $\mathcal{F}^\Gamma$ is a corresponding free energy since the isomorphism \eqref{eq:dznormal2bigphase} is $\Gamma$-equivariant. Applying $\iota^*$ to both sides of \eqref{eq:preservetau}, we obtain the last assertion, which completes the proof.
\end{proof}

\section{Examples}
\label{sec:example}
In this section, we present two examples. In the first example, we show that, for the $4$-spin CohFT, the reduction of the bihamiltonian structure agrees with the reduction established in \cite{LRZ} from the Drinfeld-Sokolov hierarchy of $A_3^{(1)}$ type to that of $C_2^{(1)}$ type. In the second example, we show that the DZ hierarchy of the orbifold Gromov-Witten theory of $\mathbb{P}^1_{2,2,2,2}$ admits two reductions. Both of them have the same dispersionless limit up to a rescaling, while their higher genus theories are not equivalent.

\subsection{Reduction of the $4$-spin theory}
\subsubsection{The $4$-spin theory}
Let $\{\Lambda_{g,n}\}$ be the CohFT defined by the $4$-spin theory. In this case, the phase space $H$ is a $3$-dimensional $\mathbb{C}$-vector space with basis $\phi_1,\phi_2,\phi_3$ and the bilinear form $\eta$ is given by
\begin{equation*}
    \eta_{\alpha\beta}=\eta(\phi_\alpha,\phi_\beta)=\delta_{\alpha+\beta,4}.
\end{equation*}
The data appearing in the homogeneity axiom \ref{item:homo} read
\begin{equation*}
    |\phi_\alpha|=\frac{\alpha-1}{4},\quad r^\alpha=0,\quad d=\frac{1}{2}.
\end{equation*}
The following properties uniquely determine the maps $\{\Lambda_{g,n}\}$\cite{PPZ}:
\begin{enumerate}[label=$\arabic*.$,ref=$\arabic*$]
    \item  $\Lambda_{g,n}(\phi_{\alpha_1}\otimes\cdots\otimes\phi_{\alpha_n})=0$ if $4\nmid 2(g-1)-n+\sum_{i=1}^n\alpha_i$,
    \item $\Lambda_{0,3}(\phi_\alpha\otimes\phi_\beta\otimes\phi_\gamma)=\delta_{\alpha+\beta+\gamma,5}$,
    \item $\Lambda_{0,4}(\phi_2,\phi_2,\phi_3,\phi_3)=\frac{1}{4}[\pt]\in H^2\left(\overline{\mathcal{M}}_{0,4},\mathbb{C}\right)$, where $[\pt]$ is the Poincar\'e dual of a point,
    \item $\Lambda_{g,n}(\phi_{\alpha_1}\otimes\cdots\otimes\phi_{\alpha_n})$ is of cohomological degree $(g-1)-\frac{n}{2}+\frac{1}{2}\sum_{i=1}^n \alpha_i$ if $4\mid 2(g-1)-n+\sum_{i=1}^n\alpha_i$.
\end{enumerate}
In particular, the Frobenius manifold of the $4$-spin theory is given by the following potential
\begin{equation*}
    F=\frac{1}{2}(t^1)^2 t^3+\frac{1}{2}t^1 (t^2)^2+\frac{1}{16}(t^2)^2(t^3)^2+\frac{1}{960}(t^3)^5.
\end{equation*}

This CohFT admits a $\mathbb{Z}_2$-symmetry as follows (see also \cite{LRZ}). Let $\gamma\in\mathbb{Z}_2$ be the nontrivial order $2$ element. We define a $\mathbb{Z}_2$-action on the phase space $H$ by
\begin{equation*}
    \gamma.\phi_1=\phi_1,\quad\gamma.\phi_2=-\phi_2,\quad\gamma.\phi_3=\phi_3.
\end{equation*}
The class $\Lambda_{g,n}(\phi_{\alpha_1}\otimes\cdots\otimes\phi_{\alpha_n})$ is nonzero only if $4\mid 2(g-1)-n+\sum_{i=1}^n\alpha_i$. In particular, we have
\begin{equation*}
    2\Big\vert\sum_{i=1}^n (\alpha_i-1),
\end{equation*}
which implies that
\begin{align*}
    \Lambda_{g,n}(\gamma.\phi_{\alpha_1}\otimes\cdots\otimes\gamma.\phi_{\alpha_n})=&(-1)^{\sum_{i=1}^n (\alpha_i-1)}\Lambda_{g,n}(\phi_{\alpha_1}\otimes\cdots\otimes\phi_{\alpha_n})\\
    =&\Lambda_{g,n}(\phi_{\alpha_1}\otimes\cdots\otimes\phi_{\alpha_n}).
\end{align*}

\subsubsection{The algebra of pseudo-differential operators}
Let $(v^{\alpha,s};\rho_\alpha^s)$ be a coordinate system on $J^\infty(\hat{\mathfrak{X}})$. The algebra of pseudo-differential operators is defined to be 
\begin{equation*}
    \mathcal{D}=\left\{\sum_{i=-\infty}^n a_i\partial^i\mid n\in\mathbb{Z},\,a_i\in\hat{\mathscr{A}}^0_v\right\}.
\end{equation*}
For two pseudo-differential operators $A=\sum_{i\leq n} a_i\partial^i$ and $B=\sum_{j\leq m} b_j\partial^j$, their product is defined by
\begin{equation*}
    A\cdot B=AB=\sum_{i\leq n}\sum_{j\leq m}\sum_{l\geq 0}\binom{i}{l}a_i\partial^l(b_j)\partial^{i+j-l}.
\end{equation*}
We define their commutator to be
\begin{equation*}
    [A,B]=AB-BA.
\end{equation*}
Let
\begin{equation*}
    (A)_{+}=\sum_{0\leq i\leq n} a_i\partial^i,\quad\res A=a_{-1}.
\end{equation*}
The operation $\res$ has an important property:
\begin{equation}
    \res [A,B]\in\partial\hat{\mathscr{A}}^0_v.
    \label{eq:ressym}
\end{equation}
Define the operator $\tr\colon\mathcal{D}\rightarrow\hat{\mathscr{F}}^0_v$ by
\begin{equation*}
    \tr(A)=\int\res A.
\end{equation*}
The property \eqref{eq:ressym} implies that $\tr(AB)=\tr(BA)$. 

\subsubsection{The bihamiltonian structure of the Drinfeld-Sokolov hierarchy of $A_3^{(1)}$ type}
Let $L=\partial^4+v^3\partial^2+v^2\partial+v^1$ be the Lax operator of the Drinfeld-Sokolov hierarchy of $A_3^{(1)}$ type \cite{DLZ2,DS}. For a local functional $F\in\hat{\mathscr{F}}^0_v$, define the variational derivative with respect to $L$ by
\begin{equation*}
    \frac{\delta F}{\delta L}=\sum_{i=1}^3\partial^{-i}\frac{\delta F}{\delta v^i}.
\end{equation*}
The brackets \eqref{eq:poisson} defined by the bihamiltonian structure of the Drinfeld-Sokolov hierarchy of $A_3^{(1)}$ type are given by
\begin{equation}
    \begin{aligned}
        \{F,G\}_1^{A_3^{(1)}}&=\tr\left(\left[B,A\right]L\right),\\
        \{F,G\}_2^{A_3^{(1)}}&=\tr\left(\left(LB\right)_{+}LA-AL(BL)_{+}+\frac{1}{4}A[L,g_B]\right),
    \end{aligned}
    \label{eq:typeA}
\end{equation}
where 
\begin{equation*}
    A=\frac{\delta F}{\delta L},\quad B=\frac{\delta G}{\delta L},
\end{equation*}
and $g_B\in\hat{\mathscr{A}}^0_v$ is any differential polynomial satisfying
\begin{equation*}
    \partial g_B=\res [L,B].
\end{equation*}
The existence of $g_B$ follows from the property \eqref{eq:ressym}.

The DZ normal coordinates for the $4$-spin DZ hierarchy are given by
\begin{equation*}
    w^\alpha=\frac{1}{(4-\alpha)(-4)^{\frac{4-\alpha-1}{2}}}\res L^{\frac{4-\alpha}{4}}.
\end{equation*}
A direct calculation shows that
\begin{align*}
    w^1&=-\frac{1}{16}v^1+\frac{1}{128}(v^3)^2+\frac{1}{32}v^{2,1}-\frac{1}{192}v^{3,2},\\
    w^2&=-\frac{\sqrt{-1}}{8}v^2+\frac{\sqrt{-1}}{8}v^{3,1},\\
    w^3&=\frac{1}{4}v^3.
\end{align*}
Let $\mathscr{P}_i^{A_3^{(1)}}$ be the matrix differential operator associated with the bracket $\{-,-\}_i^{A_3^{(1)}}$ in the DZ normal coordinate systems. Then the matrix differential operators $\mathscr{P}_i^{\DZ}$ associated with the bihamiltonian structure $P^\DZ_i$ of the $4$-spin DZ hierarchy are given by
\begin{equation*}
    \mathscr{P}_1^{\DZ}=-16\mathscr{P}_1^{A_3^{(1)}},\quad\mathscr{P}_2^{\DZ}=\mathscr{P}_2^{A_3^{(1)}}.
\end{equation*}
The nonzero components in the upper triangular part of the matrices $\mathscr{P}_i^\DZ$ are given by (see also \cite{BRS})
{
\allowdisplaybreaks
\begin{align*}
    \mathscr{P}_1^{\DZ,1,1}=&\frac{1}{48}\partial^3,\quad\mathscr{P}_1^{\DZ,1,3}=\partial,\quad\mathscr{P}_1^{\DZ,2,2}=\partial\\
    \mathscr{P}_2^{\DZ,1,1}=&\,\Bigg(\bigg(\frac{1}{32}(w^3)^3+\frac{3}{16}(w^2)^2\bigg)\partial+\bigg(\frac{3}{16}w^2 w^{2,1}+\frac{3}{64}(w^3)^2 w^{3,1}\bigg)\Bigg)\\
    &+\Bigg(\bigg(\frac{7}{256}(w^3)^2+\frac{1}{48}w^1\bigg)\partial^3+\bigg(\frac{1}{32}w^{1,1}+\frac{21}{256}w^3 w^{3,1}\bigg)\partial^2\\
    &\quad+\bigg(\frac{5}{128}(w^{3,1})^2+\frac{13}{256}w^3 w^{3,2}+\frac{1}{24}w^{1,2}\bigg)\partial\\
    &\quad+\bigg(\frac{3}{128}w^{3,1}w^{3,2}+\frac{1}{64}w^{1,3}+\frac{3}{256}w^3 w^{3,3}\bigg)\Bigg)\\
    &+\bigg(\frac{7}{1152}w^3\partial^5+\frac{35}{2304}w^{3,1}\partial^4+\frac{91}{4608}w^{3,2}\partial^3\\
    &\quad+\frac{133}{9216}w^{3,3}\partial^2+\frac{47}{9216}w^{3,4}\partial+\frac{1}{1536}w^{3,5}\bigg)+\frac{17}{36864}\partial^7,\\
    \mathscr{P}_2^{\DZ,1,2}=&\,\bigg(\frac{5}{16}w^2 w^3\partial+\frac{1}{8}w^3 w^{2,1}+\frac{1}{8}w^2 w^{3,1}\bigg)\\
    &+\bigg(\frac{7}{64}w^2\partial^3+\frac{7}{48}w^{2,1}\partial^2+\frac{17}{192}w^{2,2}\partial+\frac{1}{48}w^{2,3}\bigg),\\
    \mathscr{P}_2^{\DZ,1,3}=&\,\bigg(w^1\partial+\frac{1}{4}w^{1,1}\bigg)+\bigg(\frac{7}{192}w^3\partial^3+\frac{7}{192}w^{3,1}\partial^2\bigg)+\frac{7}{768}\partial^5,\\
    \mathscr{P}_2^{\DZ,2,2}=&\,\Bigg(\bigg(\frac{1}{8}(w^3)^2+w^1\big)\partial+\big(\frac{1}{2}w^{1,1}+\frac{1}{8}w^3 w^{3,1}\bigg)\Bigg)\\
    &+\bigg(\frac{1}{8}w^3\partial^3+\frac{3}{16}w^{3,1}\partial^2+\frac{1}{12}w^{3,2}\partial+\frac{1}{96}w^{3,3}\bigg)+\frac{1}{64}\partial^5,\\
    \mathscr{P}_2^{\DZ,2,3}=&\,\frac{3}{4}w^2\partial+\frac{1}{4}w^{2,1},\\
    \mathscr{P}_2^{\DZ,3,3}=&\,\frac{1}{2}w^3\partial+\frac{1}{4}w^{3,1}+\frac{5}{16}\partial^3.
\end{align*} 
}
The lower triangular part can be obtained from the antisymmetry property \eqref{eq:antisym}.

Let $u^1,u^2,u^3$ be the dual coordinates, which are also the DR jet variables by definition. The Miura-type transformation from the DZ normal coordinate to the DR jet variables is given in \cite{BRS} by
\begin{equation*}
    w^1=u^1+\frac{1}{96}u^{3,2},\quad w^2=u^2,\quad w^3=u^3.
\end{equation*}
This Miura transformation is clearly $\mathbb{Z}_2$-linearized and therefore, the $\mathbb{Z}_2$-action on $\hat{\mathscr{A}}^0_w$ is given by
\begin{equation}
    \gamma.w^1=w^1,\quad\gamma.w^2=-w^2,\quad\gamma.w^3=w^3.
    \label{eq:z2action}
\end{equation}
It is easy to check that the DZ bihamiltonian structure $P^\DZ_i$ satisfies \eqref{eq:main}. According to Theorem~\ref{thm:reduceddrdz}, the bihamiltonian structure of the reduced DZ hierarchy is given by $P_i^{\DZ,\mathbb{Z}_2}=\overline{\iota}^*(P_i^\DZ)$. Their corresponding matrix operators are
\begin{equation*}
    \mathscr{P}_i^{\DZ,\mathbb{Z}_2}=\begin{pmatrix}
        \mathscr{P}_i^{\DZ,1,1}&\mathscr{P}_i^{\DZ,1,3}\\
        \mathscr{P}_i^{\DZ,3,1}&\mathscr{P}_i^{\DZ,3,3}
    \end{pmatrix}\Bigg\vert_{w^{2,\bullet}\rightarrow 0}.
\end{equation*}
Here we use the canonical section
\begin{equation*}
    \hat{\mathscr{A}}_w^{\mathbb{Z}_2,0}\cong\mathbb{C}[[w^{\alpha^\prime,s}\mid \alpha^\prime\in\{1,3\}]]\hookrightarrow\hat{\mathscr{A}}_w^{0}\cong\mathbb{C}[[w^{\alpha,s}\mid \alpha\in\{1,2,3\}]],\quad w^{\alpha^\prime}\mapsto w^{\alpha^\prime}
\end{equation*}
of $\iota^*$ to identify $\hat{\mathscr{A}}_w^{\mathbb{Z}_2,0}$ with its image in $\hat{\mathscr{A}}_w^{0}$. We use the notation $(-)\vert_{w^{2,\bullet}\rightarrow 0}$ to mean setting coefficients with terms involving variables with superscript $2$ to zero in a differential operator.

\subsubsection{The bihamiltonian structure of the Drinfeld-Sokolov hierarchy of $C_2^{(1)}$ type}
Let $\widetilde{L}=\partial^4+\widetilde{v}^3\partial^2+\widetilde{v}^{3,1}\partial^1+\widetilde{v}^1$ be the Lax operator of the Drinfeld-Sokolov hierarchy of $C_2^{(1)}$ type \cite{DLZ2,DS}. For a local functional $F\in\hat{\mathscr{F}}^0_{\widetilde{v}}$, define the variational derivative with respect to $\widetilde{L}$ by
\begin{equation*}
    \frac{\delta F}{\delta \widetilde{L}}=\frac{1}{2}\left(\partial^{-1}\frac{\delta F}{\delta\widetilde{v}^1}+\frac{\delta F}{\delta\widetilde{v}^1}\partial^{-1}+\partial^{-3}\frac{\delta F}{\delta\widetilde{v}^3}+\frac{\delta F}{\delta\widetilde{v}^3}\partial^{-3}\right).
\end{equation*}
The bihamiltonian structure of the Drinfeld-Sokolov hierarchy of $C_2^{(1)}$ type is given by
\begin{align*}
    \{F,G\}_1^{C_2^{(1)}}&=\tr\left(\widetilde{L}\widetilde{B}\widetilde{A}-\widetilde{L}\widetilde{A}\widetilde{B}\right),\\
    \{F,G\}_2^{C_2^{(1)}}&=\tr\left(\left(\widetilde{L}\widetilde{B}\right)_{+}\widetilde{L}\widetilde{A}-\widetilde{A}\widetilde{L}\left(\widetilde{B}\widetilde{L}\right)_{+}\right),
\end{align*}
where
\begin{equation*}
    \widetilde{A}=\frac{\delta F}{\delta \widetilde{L}},\quad \widetilde{B}=\frac{\delta G}{\delta \widetilde{L}},
\end{equation*}
Let $\mathscr{P}_i^{C_2^{(1)}}$ be the matrix differential operator corresponding to $\{-,-\}_i^{C_2^{(1)}}$ in the following coordinate system
\begin{align*}
    \widetilde{w}^1&=-\frac{1}{16}\widetilde{v}^1+\frac{1}{128}(\widetilde{v}^3)^2+\frac{5}{192}\widetilde{v}^{3,2},\\
    \widetilde{w}^3&=\frac{1}{4}\widetilde{v}^3.
\end{align*}    
A direct calculation yields
\begin{equation}
    \left(\mathscr{P}_1^{\DZ,\mathbb{Z}_2},\mathscr{P}_2^{\DZ,\mathbb{Z}_2}\right)=\left(-16\mathscr{P}_1^{C_2^{(1)}},\mathscr{P}_2^{C_2^{(1)}}\right)\Big\vert_{\widetilde{w}^{1,\bullet}\rightarrow w^{1,\bullet},\widetilde{w}^{3,\bullet}\rightarrow w^{3,\bullet}}.
    \label{eq:z2reduction}
\end{equation}
It was shown in \cite{LRZ} that the $\mathbb{Z}_2$-reduction of the $4$-spin theory is governed by the Drinfeld-Sokolov hierarchy of $C_2^{(1)}$ type. Under this reduction, the equality \eqref{eq:z2reduction} explicitly describes the relation between the corresponding bihamiltonian structures.

\subsection{Reduction of the Gromov-Witten theory of $\mathbb{P}^1_{2,2,2,2}$}
\subsubsection{The Gromov-Witten theory of $\mathbb{P}^1_{2,2,2,2}$}
Let $\mathbb{P}^1_{2,2,2,2}$ be the Deligne-Mumford stack coarsely represented by $\mathbb{P}^1$ with exactly four distinct isotropic points of order $2$. The orbifold Gromov-Witten theory with target $\mathbb{P}^1_{2,2,2,2}$ defines a semisimple CohFT $\{\Lambda_{g,n}\}$. We refer the readers to \cite{ST} for the calculation of the associated Frobenius manifold in this example. For an introduction to the general theory, see \cite{CR1,CR2} in the symplectic-geometric setting and \cite{AGV1,AGV2} in the algebro-geometric setting.

In this case, the phase space is the Chen-Ruan cohomology $H=H^*_{\CR}\left(\mathbb{P}^1_{2,2,2,2},\mathbb{C}\right)$. As a $\mathbb{C}$-vector space, it is the usual cohomology of the rigidified cyclotomic inertia stack
\begin{equation*}
    \overline{\mathcal{I}}_\mu(\mathbb{P}^1_{2,2,2,2})=\mathbb{P}^1_{2,2,2,2}\bigsqcup\pt\bigsqcup\pt\bigsqcup\pt\bigsqcup\pt.
\end{equation*}
Here $\pt=\Spec\mathbb{C}$ and they come from four isotropic points. We can choose a basis $\phi_1,\ldots,\phi_6$ for $H$ where $\phi_1$ is the Poincar\'e dual of the fundamental class of $\mathbb{P}^1$, $\phi_6$ is the Poincar\'e dual of a point in $\mathbb{P}^1$, and $\phi_2,\phi_3,\phi_4,\phi_5$ are the Poincar\'e dual of the fundamental classes of four copies of $\pt$. They are homogeneous of degree
\begin{equation*}
    |\phi_1|=0,\quad |\phi_6|=1,\quad |\phi_2|=|\phi_3|=|\phi_4|=|\phi_5|=\frac{1}{2}.
\end{equation*}
Here, the degree is taken to be the half of the orbifold cohomological degree. The only nonzero entries of the bilinear form $\eta$ are
\begin{equation*}
    \eta(\phi_1,\phi_6)=1,\quad\eta(\phi_\alpha,\phi_\alpha)=\frac{1}{2},\quad 2\leq\alpha\leq 5.
\end{equation*}
We also have
\begin{equation*}
    d=1,\quad r^\alpha=0.
\end{equation*}
Let $\{t^{\alpha,p}\}$ be the coordinate functions for the big phase space $\mathcal{O}_{H^\infty}$. Define
\begin{align*}
    h_0(q)&=-q\frac{d}{dq}\log\left(\eta(q)\eta(q^2)^{-\frac{3}{2}}\eta(q^4)^{\frac{1}{2}}\right),\\
    h_1(q)&=-q\frac{d}{dq}\log\left(\eta(q^4)^{\frac{1}{4}}\right),\\
    h_2(q)&=-q\frac{d}{dq}\log\left(\eta(q^2)^{\frac{3}{2}}\eta(q^4)^{-\frac{3}{4}}\right),
\end{align*}
where
\begin{equation*}
    \eta(q)=q^{\frac{1}{24}}\prod_{n\geq 1}(1-q^n)
\end{equation*}
is the Dedekind eta function. The Frobenius manifold defined by this CohFT is given by
\begin{equation}
    \begin{aligned}
        F=&\frac{1}{2}(t^1)^2 t^6+\frac{1}{4}t^1\Bigl((t^2)^2+(t^3)^2+(t^4)^2+(t^5)^2\Bigr)+t^2 t^3 t^4 t^5 h_0(Q e^{t^6})\\
        &+\frac{1}{4}\Bigl((t^2)^4+(t^3)^4+(t^4)^4+(t^5)^4\Bigr)h_1(Q e^{t^6})\\
        &+\frac{1}{6}\Bigl((t^2)^2(t^3)^2+(t^2)^2(t^4)^2+(t^2)^2(t^5)^2\\
        &\qquad+(t^3)^2(t^4)^2+(t^3)^2(t^5)^2+(t^4)^2(t^5)^2\Bigr) h_2(Q e^{t^6}).
    \end{aligned}
    \label{eq:Frobeniusmanifold6}
\end{equation}
Here $Q$ is the Novikov variable.

\subsubsection{Classification of bihamiltonian structures}
We briefly recall the classification of bihamiltonian structures up to Miura-type transformations. The main references are \cite{DLZ1,LZ1}.

Let $(v^{\alpha,s},\rho_\alpha^s)$ be a coordinate system, $(P_1,P_2)$ be a bihamiltonian structure with hydrodynamic leading terms. They have the following $\varepsilon$-expansions
\begin{equation*}
    P_i=\sum_{m\geq 0}\varepsilon^{m+1}P_i^{[m]},\quad P_i^{[m]}\in\hat{\mathscr{F}}_{m+1}^2.
\end{equation*}
Denote the matrix differential operators associated with $P_i^{[m]}$ by
\begin{equation*}
    \mathscr{P}_{i,v}^{[m]}=\Bigl(\mathscr{P}_{i,v}^{\alpha\beta,[m]}\Bigr)=\Bigl(\sum_{j=0}^{m+1}\mathscr{P}_{i,v;j}^{\alpha\beta,[m]}\partial^{j}\Bigr).
\end{equation*}
The leading terms of $(\mathscr{P}_{1,v},\mathscr{P}_{2,v})$ are defined to be $(\mathscr{P}_{1,v}^{\alpha\beta,[0]},\mathscr{P}_{2,v}^{\alpha\beta,[0]})$. It is known that $\mathscr{P}_{i,v;m+1}^{\alpha\beta,[m]}$ is a symmetric (resp. antisymmetric) $(2,0)$-tensor on $\mathfrak{X}$ if $m$ is even (resp. odd). 

Suppose that there exist $\widehat{u}^k(v)\in\hat{\mathscr{A}}^0_v$, such that
\begin{align}
    &\det\Bigl(\frac{\partial \widehat{u}^k}{\partial v^\alpha}\Bigr)\Big\vert_{v^\bullet\rightarrow 0}\neq 0,
    \label{eq:cancoor}\\
    &\det(\mathscr{P}_{2,v;1}^{\alpha\beta,[0]}-\widehat{u}^k\mathscr{P}_{1,v;1}^{\alpha\beta,[0]})=0,\quad 1\leq k\leq N.
    \label{eq:eigen}
\end{align}
The condition \eqref{eq:eigen} shows that $\widehat{u}^k(v)$ are determined up to permutations. Define 
\begin{equation*}
    \widehat{u}^k_{\ori}=\widehat{u}^k\Big\vert_{v^\bullet=0}.
\end{equation*}
The condition \eqref{eq:cancoor} means $\overline{u}^k=\widehat{u}^k-\widehat{u}^k_{\ori}$ defines a coordinate system for $\mathfrak{X}$. To remain consistent with the literature, we still refer to $\widehat{u}^k$ as canonical coordinates, although they cannot define a local isomorphism, and therefore do not yield a coordinate system in our setting.

Let 
\begin{equation*}
    \mathscr{P}_{i,\overline{u}}^{kl}=\sum_{m\geq 0}\varepsilon^{m+1}\mathscr{P}_{i,\overline{u}}^{kl,[m]}=\sum_{m\geq 0}\varepsilon^{m+1}\sum_{j=0}^{m+1}\mathscr{P}_{i,\overline{u};j}^{kl,[m]}\partial^j
\end{equation*}
be the matrix differential operators associated with $(P_1,P_2)$ under the $\overline{u}$-coordinate system. We use Latin superscripts for indices in the $\overline{u}$-coordinate system, and do not adopt the Einstein summation convention. There exists $f^k\in\hat{\mathscr{A}}^0_{\overline{u}}$ such that
\begin{equation*}
    \mathscr{P}_{1,\overline{u};1}^{kl,[0]}=\delta^{k,l}f^k,\quad\mathscr{P}_{2,\overline{u};1}^{kl,[0]}=\delta^{k,l}(\overline{u}^k+\widehat{u}^k_{\ori})f^k.
\end{equation*}
Define the central invariants of $(P_1,P_2)$ to be
\begin{equation*}
    c_k(\overline{u}^k)=\frac{1}{3(f^k)^2}\left(\mathscr{P}_{2,\overline{u};3}^{kk,[2]}-(\overline{u}^k+\widehat{u}^k_{\ori})\mathscr{P}_{1,\overline{u};3}^{kk,[2]}+\sum_{l\neq k}\frac{\left(\mathscr{P}_{2,\overline{u};2}^{lk,[1]}-(\overline{u}^k+\widehat{u}^k_{\ori})\mathscr{P}_{1,\overline{u};2}^{lk,[1]}\right)^2}{(\overline{u}^l+\widehat{u}^l_{\ori}-\overline{u}^k-\widehat{u}^k_{\ori})f^l}\right).
\end{equation*}
By Remark 3.2 in \cite{BRS}, $c_k$ depends only on $\overline{u}^k$ and lies in $\hat{\mathscr{A}}^0_{\overline{u}}$.

\begin{thm}[\cite{DLZ1}, Corollary 1.11]
    Given two bihamiltonian structures with hydrodynamic type leading terms satisfying \eqref{eq:cancoor} and \eqref{eq:eigen}, there exists a Miura-type transformation of the 2nd kind transforming one to another if and only if they have the same leading terms and the same set of central invariants.
    \label{thm:classifybihamiltonian}
\end{thm}

\subsubsection{The $\mathfrak{S}_4$-reduction}
Define a $\mathfrak{S}_4$-action on $H$ by permuting $\phi_2,\phi_3,\phi_4,\phi_5$. We can choose a basis
\begin{equation*}
    \widetilde{\phi}_1=\phi_1,\quad\widetilde{\phi}_2=\frac{1}{2}(\phi_2+\phi_3+\phi_4+\phi_5),\quad\widetilde{\phi}_3=\phi_6
\end{equation*}
for $H^{\mathfrak{S}_4}$, and a basis
\begin{equation*}
    \widetilde{\phi}_4=\frac{1}{2}(\phi_2-\phi_3),\quad\widetilde{\phi}_5=\frac{1}{2}(\phi_2-\phi_4),\quad\widetilde{\phi}_6=\frac{1}{2}(\phi_2-\phi_5)
\end{equation*}
for $H^\mov$. Denote by $\{\widetilde{t}^{\alpha,p}\}$ the coordinate system for the big phase space $H^\infty$ with respect to this new basis and by $(\widetilde{w}^{\alpha,s};\widetilde{\sigma}_{\alpha}^s)$ the DZ normal coordinates.  

The Frobenius manifold defined by $\{\Lambda_{g,n}^{\mathfrak{S}_4}\}$ is given by the following potential
\begin{equation*}
     F^{\mathfrak{S}_4}=\frac{1}{2}(\widetilde{t}^1)^2 \widetilde{t}^3+\frac{1}{4}\widetilde{t}^1(\widetilde{t}^2)^2-\frac{1}{384}(\widetilde{t}^2)^4 E_2(Q e^{\widetilde{t}^3}),
\end{equation*}
where $E_2(q)=24 q\frac{d}{d q}\log\eta(q)$ is the Eisenstein series of weight $2$. The quasi-modularity of $E_2(q)$ gives rise to a Frobenius manifold structure on it. It is isomorphic,  as Frobenius manifolds to the Hurwitz space $M_{1,1}$ and to the orbit space of the complex crystallographic Coxeter group $\hat{A}_1$ (see \cite{Dub}).

Let $(\mathscr{P}_{1,\widetilde{w}}^{\mathfrak{S}_4},\mathscr{P}_{2,\widetilde{w}}^{\mathfrak{S}_4})$ be the pair of matrix differential operators corresponding to the bihamiltonian structure $(P_1^{\DR,\mathfrak{S}_4},P_2^{\DR,\mathfrak{S}_4})$ in the DZ normal coordinate system with expansions
\begin{equation*}
    \mathscr{P}_{i,\widetilde{w}}^{\mathfrak{S}_4,\alpha^\prime\beta^\prime}=\sum_{m\geq 0}\varepsilon^{m+1}\mathscr{P}_{i,\widetilde{w}}^{\mathfrak{S}_4,\alpha^\prime\beta^\prime,[m]}=\sum_{m\geq 0}\varepsilon^{m+1}\sum_{j=0}^{m+1} \mathscr{P}_{i,\widetilde{w};j}^{\mathfrak{S}_4,\alpha^\prime\beta^\prime,[m]}\partial^j.
\end{equation*}

According to Theorem 1 and Theorem 2 in \cite{DZ1}, the first few terms up to $\mathcal{O}(\varepsilon^3)$ can be described explicitly. In particular, we have
\begin{align*}
    \mathscr{P}_{1,\widetilde{w}}^{\mathfrak{S}_4,\alpha^\prime\beta^\prime,[0]}&=\eta^{\alpha^\prime\beta^\prime}\partial,\quad\mathscr{P}_{2,\widetilde{w};1}^{\mathfrak{S}_4,\alpha^\prime\beta^\prime,[0]}=E^{\zeta^\prime}c^{\alpha^\prime\beta^\prime}_{\zeta^\prime},\\
    \mathscr{P}_{i,\widetilde{w}}^{\mathfrak{S}_4,\alpha^\prime\beta^\prime,[1]}&=0,\quad i=1,2\\
    \mathscr{P}_{1,\widetilde{w};3}^{\mathfrak{S}_4,\alpha^\prime\beta^\prime,[2]}&=\frac{1}{12}\eta^{\mu\nu}c^{\alpha^\prime\beta^\prime}_{\mu\nu},\\
    \mathscr{P}_{2,\widetilde{w};3}^{\mathfrak{S}_4,\alpha^\prime\beta^\prime,[2]}&=\frac{1}{12}\left(\bigl(\frac{3}{2}-|\widetilde{\phi}_\nu|\bigr)c^{\mu^\prime\nu}_\nu c^{\alpha^\prime\beta^\prime}_{\mu^\prime}+E^{\zeta^\prime}c^{\mu^\prime\nu}_{\zeta^\prime\nu}c^{\alpha^\prime\beta^\prime}_{\mu^\prime}+E^{\zeta^\prime}c^{\mu\nu}_{\zeta^\prime}c^{\alpha^\prime\beta^\prime}_{\mu\nu}\right),
\end{align*}
where
\begin{align*}
    &(E^1,E^2,E^3)=\left(\widetilde{t}^1,\frac{1}{2}\widetilde{t}^2,0\right),\\
    &c^{\alpha^\prime\beta^\prime}_{\mu}=\eta^{\alpha^\prime\nu^\prime}\eta^{\beta^\prime\zeta^\prime}\frac{\partial^3 F(t(\widetilde{t}))}{\partial \widetilde{t}^{\mu}\partial \widetilde{t}^{\nu^\prime}\partial \widetilde{t}^{\zeta^\prime}}\Bigg\vert_{\widetilde{t}^{\xi^\prime}\rightarrow\widetilde{w}^{\xi^\prime},\widetilde{t}^{\xi^\pprime}\rightarrow 0},\\
    &c^{\alpha^\prime\beta^\prime}_{\mu\nu}=\eta^{\alpha^\prime\lambda^\prime}\eta^{\beta^\prime\zeta^\prime}\frac{\partial^4 F(t(\widetilde{t}))}{\partial \widetilde{t}^{\mu}\partial \widetilde{t}^{\nu}\partial \widetilde{t}^{\lambda^\prime}\partial \widetilde{t}^{\zeta^\prime}}\Bigg\vert_{\widetilde{t}^{\xi^\prime}\rightarrow\widetilde{w}^{\xi^\prime},\widetilde{t}^{\xi^\pprime}\rightarrow 0}.
\end{align*}
The canonical coordinates are
\begin{align*}
    \widehat{u}^1=&\widetilde{w}^1+\frac{3}{2}(\widetilde{w}^2)^2h_1(Q^{\frac{1}{2}} e^{\frac{\widetilde{w}^3}{2}})-\frac{1}{6}(\widetilde{w}^2)^2h_2(Q^{\frac{1}{2}} e^{\frac{\widetilde{w}^3}{2}}),\\
    \widehat{u}^2=&\widetilde{w}^1-\frac{1}{2}(\widetilde{w}^2)^2h_0(Q^{\frac{1}{2}} e^{\frac{\widetilde{w}^3}{2}})+\frac{1}{3}(\widetilde{w}^2)^2h_2(Q^{\frac{1}{2}} e^{\frac{\widetilde{w}^3}{2}}),\\
    \widehat{u}^3=&\widetilde{w}^1+\frac{1}{2}(\widetilde{w}^2)^2h_0(Q^{\frac{1}{2}} e^{\frac{\widetilde{w}^3}{2}})+\frac{1}{3}(\widetilde{w}^2)^2h_2(Q^{\frac{1}{2}} e^{\frac{\widetilde{w}^3}{2}}).
\end{align*}
One can directly verify that the central invariants of $(P_1^{\DR,\mathfrak{S}_4},P_2^{\DR,\mathfrak{S}_4})$ are $\{\frac{1}{6},\frac{1}{24},\frac{1}{24}\}$.

According to Theorem 1 in \cite{LWZ1} and the main theorem in \cite{LWZ2}, the tau function of the reduced DZ hierarchy does not satisfy the Virasoro constraints constructed in \cite{DZ2}, even up to Miura-type transformations. Instead, the correct Virasoro symmetries are obtained by adding nonlinear differential polynomial terms.

\subsubsection{The $\mathbb{Z}_2\times\mathbb{Z}_2$-reduction}
\label{sssec:equalcentralinvariants}
Define a $\mathbb{Z}_2\times\mathbb{Z}_2$-action on $H$ as follows. We choose two distinct order $2$ elements $\gamma_1,\gamma_2$ in $\mathbb{Z}_2\times\mathbb{Z}_2$. Let $\gamma_1$ act on $H$ by permuting $\phi_2$ and $\phi_3$, and let $\gamma_2$ act by changing the signs of $\phi_4,\phi_5$. We choose a basis
\begin{equation*}
    \overline{\phi}_1=\phi_1,\quad\overline{\phi}_2=\frac{1}{2}\phi_2+\frac{1}{2}\phi_3,\quad\overline{\phi}_3=\frac{1}{2}\phi_6
\end{equation*}
for $H^{\mathbb{Z}_2\times\mathbb{Z}_2}$ and a basis
\begin{equation*}
    \overline{\phi}_4=\phi_2-\phi_3,\quad\overline{\phi}_5=\phi_4,\quad\overline{\phi}_6=\phi_5
\end{equation*}
for $H^\mov$. Denote the dual basis for the big phase space $H^\infty$ with respect to this new basis by $\{\overline{t}^{\alpha,p}\}$. The Frobenius manifold is defined by
\begin{equation}
    F^{\mathbb{Z}_2\times\mathbb{Z}_2}=\frac{1}{2}\left(\frac{1}{2}(\overline{t}^1)^2\overline{t}^3+\frac{1}{4}\overline{t}^1(\overline{t}^2)^2-\frac{1}{384}(\overline{t}^2)^4 E_2(\overline{Q} e^{\overline{t}^3})\right),
    \label{eq:Frobeniusmanifold222}
\end{equation}
which coincides with $\frac{1}{2}F^{\mathfrak{S}_4}$ where $\overline{Q}=Q^2$. One can verify similarly that the central invariants in this case are given by $\{\frac{1}{12},\frac{1}{12},\frac{1}{12}\}$. We further rescale $\varepsilon$ by $\varepsilon\mapsto\frac{1}{2}\varepsilon$ so that all central invariants of the
$\mathbb{Z}_2\times\mathbb{Z}_2$-reduced DZ hierarchy become equal to $\frac{1}{24}$ and its bihamiltonian structure has the same leading terms as that of the $\mathfrak{S}_4$-reduced DZ hierarchy.
\begin{cor} 
    The rescaled $\mathbb{Z}_2\times\mathbb{Z}_2$-reduced DZ hierarchy cannot be transformed into the $\mathfrak{S}_4$-reduced DZ hierarchy by a Miura-type transformation.
\end{cor}

\begin{proof}
    This follows from the fact that they have unequal central invariants (see Theorem~\ref{thm:classifybihamiltonian}).
\end{proof}

By the main theorem in \cite{LWZ2}, the rescaled $\mathbb{Z}_2\times\mathbb{Z}_2$-reduced DZ hierarchy together with its tau structure defines a tau function satisfying the usual Virasoro constraints after a Miura-type transformation. On the other hand, the topological recursion of the genus $1$ spectral curve also determines a tau function satisfying the usual Virasoro constraints, whose genus-zero data coincide with the Hurwitz space $M_{1;1}$ (see \cite{GJZ}), and hence with $F^{\mathbb{Z}_2\times\mathbb{Z}_2}$. The Virasoro constraints force them to coincide up to an automorphism of the big phase space induced by a Miura-type transformation on $\hat{\mathscr{A}}$, together with a multiplication by the exponential of an affine linear polynomial.

\begin{cor}
    The Miura-type transformation above must be nontrivial.
    \label{cor:nontrivial}
\end{cor}

\begin{proof}
    Since the tau function determined by topological recursion satisfies the usual Virasoro constraints, it must be determined by Dubrovin-Zhang's loop equation (see \cite{DZ3}). In particular, the genus one primary free energy is given by the $G$-function of the Frobenius manifold \eqref{eq:Frobeniusmanifold222} (see \cite{DZ1,DZ2}):
    \begin{equation*}
        -\log\left[(\overline{t}^2)^{\frac{1}{8}}\eta(\overline{Q}e^{\overline{t}^3})\right],
    \end{equation*}
    which has a logarithmic singularity. The $G$-function in this case was computed in \cite{DZ1}.
    
    On the other hand, the genus one primary free energy of $\{\Lambda_{g,n}^{\mathbb{Z}_2\times\mathbb{Z}_2}\}$ is given by the restriction of the $G$-function of the Frobenius manifold \eqref{eq:Frobeniusmanifold6} (see \cite{ST}):
    \begin{equation*}
        \mathcal{F}_1^{\mathbb{Z}_2\times\mathbb{Z}_2}\Big\vert_{\overline{t}^{\bullet,>0}\rightarrow 0}=\iota^{\infty,*}\left(\mathcal{F}_1\Big\vert_{t^{\bullet,>0}\rightarrow 0}\right)=-\frac{1}{2}\log\eta(\overline{Q}e^{\overline{t}^3}).
    \end{equation*}
    Consequently, there must exist a nontrivial Miura-type transformation that eliminates the logarithmic singularity involving $\overline{t}^2$.
\end{proof}

\vskip 0.3cm
\noindent \textbf{Acknowledgements.}
This work is supported by NSFC No.\,12571266. We would like to thank Ce Ji and Sergey Shadrin for very helpful discussions and comments on this work.

\end{document}